\documentclass[12pt, leqno]{article}

\usepackage{amsmath, amssymb, amscd, verbatim, xspace,amsthm}
\usepackage{latexsym, epsfig, color}

\newcommand{\Z}{\ensuremath{{\mathbb{Z}}}\xspace}
\renewcommand{\P}{\ensuremath{{\mathbb{P}}}}
\newcommand{\Q}{\ensuremath{{\mathbb{Q}}}}

\newcommand{\ra}{\rightarrow}
\newcommand{\lra}{\longrightarrow}

\newcommand\Hom{\operatorname{Hom}}

\newcommand\Sym{\operatorname{Sym}}
\newcommand\sym{\operatorname{sym}}

\newcommand\tensor{\otimes}
\newcommand\isom{\cong}

\newcommand\tesnor{\otimes}

\newcommand\Disc{\operatorname{Disc}}
\newcommand\GL{\operatorname{GL}}

\newcommand\SL{\operatorname{SL}}

\newcommand\Spec{\operatorname{Spec}}
\newcommand\Proj{\operatorname{Proj}}

\newcommand\wt{\wedge^2}
\newcommand\OSf{\ensuremath{\mathcal{O}_\Sf}\xspace}

\newcommand\Rf{\ensuremath{R_f}\xspace}
\newcommand\If{\ensuremath{{I_f}}\xspace}
\newcommand\GZ{\ensuremath{\GL_2(\Z)}\xspace}
\newcommand\Sf{\ensuremath{{S_f}}\xspace}
\newcommand\BR{\ensuremath{B}\xspace}
\newcommand\BS{\ensuremath{S}\xspace}
\newcommand\fRf{\ensuremath{{\mathcal{R}_f}}\xspace}
\newcommand\fR[1]{\ensuremath{{\mathcal{R}_{#1}}}\xspace}
\newcommand\fIf{\ensuremath{{\mathcal{I}_f}}\xspace}
\newcommand\fI[1]{\ensuremath{{\mathcal{I}_{#1}}}\xspace}
\newcommand\Qf{\ensuremath{Q_f}\xspace}
\newcommand\ang[1]{\left\langle #1 \right\rangle}

\newcommand\OS{\ensuremath{{\mathcal{O}_\BS}}\xspace}

\newcommand\map[4]{\ensuremath{\begin{array}{ccc}#1&\lra&#2\\#3&\mapsto&#4\end{array}}}

\newcommand\bq{\begin{equation}}
\newcommand\eq{\end{equation}}

\newtheorem{proposition}{Proposition}[section]
\newtheorem{theorem}[proposition]{Theorem}
\newtheorem{corollary}[proposition]{Corollary}
\newtheorem{example}[proposition]{Example}

\newtheorem{lemma}[proposition]{Lemma}

\theoremstyle{remark}
\newtheorem{remark}[proposition]{Remark}
\usepackage[all,cmtip]{xy}

\usepackage{fullpage}

\newenvironment{definition}{\vspace{2 ex}{\noindent{\bf Definition. }}}{\vspace{2 ex}}
\newenvironment{notation}{\vspace{2 ex}{\noindent{\bf Notation. }}}{\vspace{2 ex}}

\title{Rings and ideals parametrized by binary $n$-ic forms}

\author{Melanie Matchett Wood\thanks{mwood@math.stanford.edu}\\
Stanford University, Department of Mathematics, \\
Building 380, Sloan Hall, \\
Stanford, California 94305}

\begin{document}

\maketitle

\begin{abstract}
The association of algebraic objects to forms has had many important applications in number theory.  Gauss, over two centuries ago, studied quadratic rings and ideals associated to binary quadratic forms, and found that ideal classes of quadratic rings are exactly parametrized by equivalence classes of integral binary quadratic forms.
Delone and Faddeev, in 1940, showed that cubic rings are parametrized by equivalence classes of integral binary cubic forms.
Birch, Merriman, Nakagawa, Corso, Dvornicich, and Simon have all studied rings
associated to binary forms of degree $n$ for any $n$, but it has not previously been known which rings, and with what additional structure, are associated to binary forms. 

In this paper, we show exactly what algebraic structures are parametrized  by binary $n$-ic forms, for all $n$.  The algebraic data associated to an integral binary $n$-ic form includes a ring isomorphic to $\Z^n$ as a $\Z$-module, an ideal class for that ring, and a condition on the ring and ideal class that comes naturally from geometry.  
In fact, we prove these parametrizations when any base scheme replaces the integers, and show that the correspondences between forms and the algebraic data are functorial in the base scheme.  We give geometric constructions
of the rings and ideals from the forms that parametrize them and a simple construction of the form from an appropriate ring and ideal.  
\end{abstract}

\section{Introduction}\label{intro}
When one looks for a parametrizing space for degree $n$ number fields, binary $n$-ic forms are a natural guess.  It turns
out that for $n=3$ this guess is correct.  We have that
$\GL_2(\Q)$ classes of binary cubic forms with rational coefficients are in bijection with isomorphism classes of cubic $\Q$-algebras and irreducible forms correspond to cubic number fields.  Moreover,
an analogous result allows the parametrization of orders in those number fields; 
 $\GL_2(\Z)$ classes of integral binary cubic forms are in bijection with isomorphism classes of cubic rings (\cite{DF}, see also \cite{DH}, \cite{GGS}, \cite{HCL3}).
For other $n$, the space of binary $n$-ic forms parametrizes algebraic data that is
more subtle than this.
It has long been known that binary quadratic forms parametrize ideal classes in quadratic rings
(originally in \cite{Gauss}, see \cite{Towber}, \cite{Kneser}, or \cite{BinQuad} for a modern treatment).
In this paper, we construct the algebraic data associated to a binary $n$-ic form, and determine what algebraic structures are in fact parametrized by binary $n$-ic forms for all $n$.

Every binary $n$-ic form with integral coefficients does have an associated ring.  The rings that come from binary $n$-ic forms are interesting for many reasons in their own right, in particular because we have several other tools to understand these rings.  Del Corso,  Dvornicich, and Simon have viewed the rings associated to binary $n$-ic forms as a generalization of monogenic rings and have described how a prime splits in a ring associated to a binary $n$-ic form in terms of the factorization of the form modulo the prime \cite{PrimeSplit}.  They have also given a condition on the form equivalent to the $p$-maximality of the associated ring.
Simon \cite{Simonobst} uses the ring associated to a binary form to find a class group obstruction equations of the form
$Cy^d=f(x,y)$ having integral solutions (where $f$ is the binary form).  
  Work of the author finds an explicit moduli space for ideal classes in the rings associated to binary $n$-ic forms \cite{2nn}.  Thus, we can work explicitly with these rings, prime splitting in them, and their ideal classes.

However, in addition to the ring that is canonically associated to a binary form, there is more associated data, including ideal classes of the ring.
Some of these ideal classes have been constructed for irreducible, primitive forms in \cite{SimonIdeal}, \cite{PrimeSplit},
and \cite{Simonobst}.
 In Section \ref{S:ZConst}, we give four different ways to construct the associated ring and ideal classes 
from a binary form
1) explicitly as a subring of a $\Q$-algebra, 
2) by giving the multiplication and action tables, 3) via a simple geometric construction that works when $f\not\equiv0$, and 
4) via a more complicated geometric construction that works in all cases.  
The geometric constructions answer a question posed by Lenstra at the Lorentz Center Rings of Low Rank Workshop in 2006 about giving
a basis-free description of the ring associated to a binary form.
 In the case $n=3$, the geometric construction of 4) was originally given in a letter of Deligne \cite{Delcubic}.
We see that for $n\ne 2$ the ring associated to a form is Gorenstein if and only if the form is primitive.  Also, the ideal classes associated to the form are invertible if and only if
the form is primitive.  The geometric construction of a ring and ideal classes from a binary form is so simple that we give it here.  

A binary $n$-ic form with integer coefficients describes a subscheme of
$\P^1_\Z$ which we call $\Sf$.  Let $\mathcal{O}(k)$ denote
the usual sheaf on $\P^1_\Z$ and let
$\mathcal{O}_\Sf(k)$ denote its pullback to \Sf.
Also, for a sheaf $\mathcal{F}$, let
$\Gamma(\mathcal{F})$ be the global sections of $\mathcal{F}$.
When $f\not\equiv 0$, the ring associated to the binary $n$-ic form $f$ is simply the ring of global functions of $\Sf$.
The global sections 
$\Gamma(\mathcal{O}_{\Sf}(k))$ have an $\Gamma(\mathcal{O}_{\Sf})$-module structure,
and for a binary form $f\not\equiv 0$
and
$-1\leq k\leq n-1$, the global sections
$\Gamma(\mathcal{O}_{\Sf}(k))$ give a module of the ring associated to $f$ which is realizable as an ideal class.
When $n=2$, taking $k=1$ we obtain the ideal associated to the binary quadratic form in Gauss composition.
(This construction gives an ideal even when $f$ is reducible or non-primitive.  See \cite{BinQuad} for a complete description of the
$n=2$ case.)  When $n=3$, 
we expect to obtain canonical modules for the ring since we know binary cubic forms parametrize exactly cubic rings.
When $n=3$, by taking $k=1$ we obtain the inverse different of the ring associated to the binary cubic form, and in general
taking $k=n-2$ gives the inverse different (see Theorem~\ref{T:invdif}).
Thus from a binary form, we naturally construct a ring and several ideal classes.  As we are interested in understanding exactly what data is parametrized by binary forms, the natural questions remaining are: is there more data naturally associated to the form?  is some of the data we have already constructed redundant, in other words could it be constructed from other pieces of the data?  and what rings and ideal classes actual arise from this construction?

First, we will see that there is more important structure to the ring and ideal classes we have constructed.
Given a form $f$, let $R$ be the associated ring, and $I$ the ideal from $k=n-3$. 
From the exact sequences on $\P^1_\Z$
$$
0\ra \mathcal{O}(-n) \stackrel{f}{\ra} \mathcal{O} \ra \mathcal{O}/f(\mathcal{O}(-n)) \ra 0$$
and
$$
0\ra \mathcal{O}(-3) \stackrel{f}{\ra} \mathcal{O}(n-3) \ra \mathcal{O}(n-3)/f(\mathcal{O}(-3)) \ra 0
$$
we obtain exact sequences
$$
0\ra \Z \ra R \ra H^1(\P^1_\Z,\mathcal{O}(-n)) \ra 0 $$
and
$$
0\ra H^0(\P^1_\Z,\mathcal{O}(n-3)) \ra I \ra H^1(\P^1_\Z,\mathcal{O}(-3)) \ra 0.
$$
We have a map $R\tensor I \ra I$ from the action of the ring on the ideal, and thus
a map $\phi: R/\Z \tesnor H^0(\P^1_\Z,\mathcal{O}(n-3)) \ra H^1(\P^1_\Z,\mathcal{O}(-3))$.  It is easy to see that
with the identification of $R/\Z$ with $H^1(\P^1_\Z,\mathcal{O}(-n))$, that $\phi$ is the same as the natural map
$$
H^1(\P^1_\Z,\mathcal{O}(-n)) \tesnor H^0(\P^1_\Z,\mathcal{O}(n-3)) \ra H^1(\P^1_\Z,\mathcal{O}(-3)).
$$
Note if we write $V=\Z^2$, then we have $H^1(\P^1_\Z,\mathcal{O}(-n)))=\Sym_{n-2} V^*$, and
$H^0(\P^1_\Z,\mathcal{O}(n-3))=\Sym^{n-3} V$, and $H^1(\P^1_\Z,\mathcal{O}(-3))=V^*$.

In Section \ref{S:TwistedMain}, we prove that the above algebraic data is precisely the data parametrized by binary $n$-ic forms.
\begin{theorem}
 Given a ring $R$ and an $R$-module $I$, we have that $R$ and $I$
are associated to a binary $n$-ic form if and only if we can write $R/\Z = \Sym_{n-2} V^*$
and an exact sequence $0\ra \Sym^{n-3} V \ra I \ra V^* \ra 0$ such that the map
$\Sym_{n-2} V^* \tesnor \Sym^{n-3} V  \ra V^*$ given by the action of $R$ on $I$ is the same as the natural map
between those $\Z$-modules.
It is equivalent to require that $R$ have a $\Z$-module basis $\zeta_0=1,\zeta_1,\dots,\zeta_{n-1}$ and
and $I$ have a $\Z$-module basis $\alpha_1,\alpha_2,\beta_1,\dots,\beta_{n-2}$ such that
$$
\text{the $\alpha_i$ coefficient of $\zeta_j\beta_k$ is }  \begin{cases} 
1 &\text{if $i + j +k = n+1$ }\\
0 &\text{otherwise.}
\end{cases}
$$
\end{theorem}

The equivalence can be computed by  working out the natural map $\Sym_{n-2} V^* \tesnor \Sym^{n-3} V  \ra V^*$ in terms
of an explicit basis.
It is easy to see that when $n=3$ this condition requires that $I$ is isomorphic to $R$ as an $R$-module.
So we see here that only one of the ideal classes constructed is really new data, since the binary form, and thus all its associated ideal classes,
can be recovered from $R$, $I$, and the exact sequence above.

All of the work in the paper can be done with an arbitrary base scheme (or ring) replacing $\Z$ in the above, and we now state
a precise theorem capturing the above claims over an arbitrary base.
Let $S$ be a scheme, and $\OS$ its sheaf of regular functions.
A \emph{binary $n$-ic form} over $S$ is a locally free rank $2$ $\OS$-module $V$, and an element $f\in \Sym^n V$.
An $l$-twisted binary $n$-ic form over $S$ is a locally free rank $2$ $\OS$-module $V$, and an element $f\in \Sym^n V\tesnor
(\wedge^2 V)^{\tesnor l}$.
A \emph{binary $n$-pair} is an $\OS$-algebra $R$, an $R$-module $I$, an
exact sequence  $0\ra \Sym^{n-3} Q^* \ra I \ra Q \ra 0$ such that $Q$ is a locally free rank $2$ $\OS$-module, and
an isomorphism $R/\OS\isom \Sym_{n-2} Q$ that identifies the map
$R/\OS \tesnor  \Sym^{n-3} Q^* \ra Q$ induced from the action of $R$ on $I$  with the natural map
$\Sym_{n-2} Q \tesnor \Sym^{n-3} Q^*  \ra Q$. 
In Section~\ref{S:SConst}, we give a geometric construction of rings and modules from (twisted) binary $n$-ic forms over a scheme $S$, motivated
by the geometric description given above over $\Z$.
Our main theorem is the following, proved in Section~\ref{S:TwistedMain}.
\begin{theorem}\label{T:TwistedMain}
For $n\geq 3$, we have a bijection between $(-1)$-twisted binary $n$-ic forms over $S$ and binary $n$-pairs over $S$, and the bijection commutes with base change in S.  In other words, we have a isomorphism of the moduli stack of $(-1)$-twisted binary $n$-ic forms and the moduli stack of binary $n$-pairs.
\end{theorem}

Analogs of Theorem~\ref{T:TwistedMain} can be proven for $l$-twisted binary forms for all $l$.  We have already given the idea of a geometric construction for one direction of the bijection in Theorem~\ref{T:TwistedMain} (see Section~\ref{S:SConst} for the details), and we now give a simple construction of the other direction of the bijection.
We can write the construction of a $(-1)$-twisted binary $n$-ic form from a 
binary $n$-pair as the evaluation
$$x \mapsto x\wedge \phi(x^{n-2})x$$
of he above degree $n$ map $Q \ra \wt Q$, where
$\phi$ is the isomorphism $\Sym_{n-2} Q\isom R/\OS$ and we lift 
$x$ to the ideal $I$ to act on it by $R$ and then take the quotient to $Q$.
It is not clear a priori that this map is even well-defined, but that will follow
from the definition of a binary $n$-pair (Lemma~\ref{L:zwd}).

\section{Constructing a ring and modules from a binary $n$-ic form over \Z}\label{S:ZConst}
\subsection{Concrete construction}\label{S:RI}
In this section we will explicitly realize the ring and ideals associated to a binary $n$-ic form
inside a $\Q$-algebra.
Given  a \emph{binary $n$-ic form}, 
$$
f_0x^n +f_1 x^{n-1}y +\dots +f_ny^n \quad \textrm{with}\quad f_i\in\Z,
$$
such that $f_0\ne 0$, we can form a ring
$\Rf$ as a subring of $\Qf:=\Q(\theta)/(f_0\theta^n +f_1 \theta^{n-1} +\dots +f_n)$ 
with $\Z$-module basis
 \begin{gather}\label{E:basis}
 \zeta_0=1\\\notag
 \zeta_1=f_0\theta\\\notag
 \zeta_2=f_0\theta^2+f_1\theta\\\notag
 \vdots\\\notag
 \zeta_k=f_0\theta^k+\dots+f_{k-1}\theta\\\notag
 \vdots\\\notag
 \zeta_{n-1}=f_0\theta^{n-1}+\dots+f_{n-2}\theta.
  \end{gather}
Since $f_0\ne 0$, we have that $\Rf$ is a free rank $n$ \Z-module, i.e. \emph{a rank $n$ ring} in the terminology of Bhargava \cite{HCL3}.
Birch and Merriman \cite{BM} studied 
 this $\Z$-submodule of  $\Qf$,
and Nakagawa \cite[Proposition 1.1]{Naka} proved that this  $\Z$-submodule is a ring
(though Nakagawa worked only with irreducible $f$, his proof makes sense for all $f$).  
 Nakagawa writes down the multiplication table of $\Rf$ explicitly as follows:
\begin{equation}\label{E:mult}
\zeta_i\zeta_j=-\sum_{\max(i+j-n,1)\leq k\leq i} f_{i+j-k}\zeta_k
+\sum_{j< k\leq \min(i+j,n)} f_{i+j-k}\zeta_k
\quad \textrm{for } 1\leq i,j\leq n-1,
\end{equation}
where $\zeta_n:=-f_n$.
If $f_0=0$, we could still use the above multiplication table to define a rank $n$ ring (see Section~\ref{S:tables}).
 We have the discriminant equality 
$\Disc\Rf=\Disc f$
(see, for example, \cite[Proposition 4]{Simon}),
which is a point of interest in $\Rf$ in previous works (e.g. \cite{Naka}, \cite{Simon}).  

\begin{remark}
Throughout this paper, it will be useful to also make the above construction with $\Z$ replaced by
$\Z[f_0,\dots f_n]$, where the $f_i$ are formal variables, and with  $f=f_0x^n + \dots +f_ny^n$, which we call the \emph{universal form.}
If $K$ is the fraction field of $\Z[f_0,\dots f_n]$, we can then work in
$K(\theta)/(f_0\theta^n +f_1 \theta^{n-1} +\dots +f_n)$ instead of 
$\Q(\theta)/(f_0\theta^n +f_1 \theta^{n-1} +\dots +f_n)$.  The multiplication table in
Equation~\eqref{E:mult} still holds, as Nakagawa's proof can also be interpreted in this context.
\end{remark}

   When $f_0\ne 0$, we can also form
 a fractional ideal $\If=(1,\theta)$ of \Rf (lying in $ \Qf$).
There is a natural $\GL_2(\Z)$ action on binary forms, and the ring
 $\Rf$ and the ideal class of $\If$ are invariant under this action 
 The invariance will follow from our geometric construction of this ideal in Section~\ref{S:ZGeomConst}.
 (See also \cite[Proposition 1.2]{Naka} for a direct proof of the invariance of $\Rf$, and \cite[Th\'eor\`eme 3.4]{Simon} 
 which, in the case when $f$ is irreducible and primitive,
  considers a sequence of ideals $\mathfrak{J}_j$, all in the ideal class of 
 $\If$, and
 proves that this ideal class is $\SL_2$
 invariant.)
   The powers of \If give a sequence of ideals $\If^0,\If^1,\dots,\If^{n-1},\dots$
 whose classes are each \GZ invariant.
  We can write down the following explicit $\Z$-module bases for $\If^k$ for $0\leq k\leq n -1$:
 \begin{align}\label{E:Ibasis}
 \If^k =\langle
  1,\theta,\dots,\theta^k,\zeta_{k+1},\dots,\zeta_{n-1}\rangle_\Z,
 \end{align}
 where $\langle s_1,\dots, s_n \rangle_R$ denotes the 
 $R$-module generated by $s_1,\dots,s_n$.
 Equivalently to Equation~\eqref{E:Ibasis}, we have for $0\leq k\leq n -1$
 \begin{equation}\label{E:basis2}
   \If^k=\langle
  1,\theta,\dots,\theta^k,f_0\theta^{k+1},f_0\theta^{k+2}+f_1\theta^{k+1},\dots,
 f_0\theta^{n-1}+f_1\theta^{n-2}+\dots+f_{n-k-2}\theta^{k+1}\rangle_\Z
 .
\end{equation}
 To be clear, we give the boundary cases explicitly:
 \begin{gather*}
 \If^{n-2}=\langle 1,\theta,\dots,\theta^{n-2},f_0\theta^{n-1}\rangle_\Z\\
 \If^{n-1}=\langle 1,\theta,\dots,\theta^{n-1}\rangle_\Z.
 \end{gather*}
  Proposition~\ref{P:Ibasis} in the Appendix (Section~\ref{S:App1}) shows
  that the $\Z$-modules given above are equal to the ideals
  we claim.  Clearly, the given $\Z$-modules are subsets of the respective ideals
 and contain the ideal generators, and so
 it only remains to check that the given $\Z$-modules are actually ideals.

 If we look at the $\Z$ bases of $\If^2$, $\If^1$, and $\If^0$ given in Equation~\eqref{E:basis2}, they naturally
 lead to considering another $\Z$-module (given by Equation~\eqref{E:basis2} when $k=-1$)
\begin{equation}\label{E:Isbasis}
 \If^\#=\langle f_0,f_0\theta+f_1,\dots,
 f_0\theta^{n-1}+f_1\theta^{n-2}+\dots+f_{n-1}\rangle_\Z.
\end{equation}
It turns out that $\If^\#$ is an ideal of $\Rf$, which is shown in  
 Proposition~\ref{P:Isbasis} in the Appendix (Section~\ref{S:App1}).
This ideal is studied in the case of $f$ irreducible and primitive 
as $\mathfrak{b}$ in \cite{Simon} and \cite{Simonobst}
and as $\mathfrak{B}$ in \cite{PrimeSplit}.

\begin{remark}
 Similarly, we can form the fractional ideals $I_f^k$ and $I_f^\#$  over the base ring
$\Z[f_0,\dots f_n]$ and with $f=f_0x^n + +\dots +f_ny^n$, working in
$K(\theta)/(f_0\theta^n +f_1 \theta^{n-1} +\dots +f_n)$.  
The ideals have $\Z[f_0,\dots f_n]$-module bases
as given in
Equations~\eqref{E:Ibasis}, \eqref{E:basis2}, and \eqref{E:Isbasis}, and these
$\Z[f_0,\dots f_n]$-modules are $R$ ideals by the same proofs as in the $\Z$ case.
\end{remark}

Given the sequence
$\If^2,\If^1,\If^0$ that led us to define $\If^\#$ one might expect that
$\If^\#$ is the same as $\If^{-1}$. 
However, it turns out that \If is not always invertible.
We do have the following proposition
(proven in Proposition~\ref{P:invertextra} of the Appendix (Section~\ref{S:App1})).  A 
form $f$ is \emph{primitive} if its coefficients generate the unit ideal in $\Z$.
\begin{proposition}\label{P:Iinvert}
For $f\not\equiv 0$, the ideal class of $\If$ is invertible if and only if the form $f$ is primitive.
Also, the ideal class of $\If^\#$ is invertible if and only if the form $f$ is primitive.
In the
case that $f$ is primitive, $\If^{-1}=\If^\#$.
\end{proposition}

When $f$ is primitive, Simon \cite[Proposition 3.2]{SimonIdeal} proved that
the ideal classes of what we call $\If$ and $\If^\#$ are inverses.  
Of course, for any $k>0$, we have $\If^k$ is invertible if and only if
$\If$ is.  Some of the ideal classes $\If^k$ are particularly interesting.
For example, we have the following result which we prove in Corollary~\ref{C:JInvdif}.
\begin{theorem}\label{T:invdif}
The class of $\If^{n-2}$ is the class of the inverse different of $\Rf$.
In other words, as $\Rf$ modules, $\If^{n-2}\isom \Hom_\Z (\Rf,\Z)$.
\end{theorem}
Simon \cite[Proposition 14]{Simonobst} independently discovered that when $f$ is primitive and irreducible
that $(\If^\#)^{2-n}$ is in the ideal class of the inverse different of $\Rf$.  In this paper,
we find that while $(\If^\#)^{2-n}$ is not naturally constructed as a module, $\If^{n-2}$ can
be naturally constructed and is always the inverse different, even when $f$ is reducible, primitive, or the zero form!
When $f\equiv0$, we construct $\If^{n-2}$ as a module and the above theorem holds,
but the module is not realizable as a fractional ideal of $\Rf$.

\begin{corollary}\label{C:Goren}
For $n\ne2$ and $f\not\equiv 0$,
 the ring $\Rf$ is Gorenstein if and only if the form $f$ is primitive.
\end{corollary}
\begin{proof}
It is known that for rank $n$ rings, the condition of Gorenstein
is equivalent to the inverse different being invertible.  For the ring
$\Rf$, the inverse different is in the same ideal class as $\If^{n-2}$
and thus this follows from Proposition~\ref{P:Iinvert}.
\end{proof}

\begin{remark}
When we have a binary form with $f_0=\pm 1$,
then $\Rf=\Z[\theta]/f(\theta)$.  Such rings, generated by one
element, are called \emph{monogenic}.  We see that all monogenic rings
are $\Rf$ for some binary form $f$.
Also, in this case $\If^k\isom I_f^\#\isom \Rf$ as $\Rf$-modules.
\end{remark}

\subsection{Explicit multiplication and action tables}\label{S:tables}

If a form $f=f_0x^n +f_1 x^{n-1}y +\dots +f_ny^n$ has $f_0=0$, but
$f\not\equiv 0$, then we can act by \GZ to take 
$f$ to a form $f'$ with $f_0'\ne 0$.  We then can define the ring $\Rf$
and the $\Rf$ ideal classes $\If$ and $\If^\#$ using $f'$.  Since
the ring and ideal classes are \GZ invariants, it does 
not matter which $f'$ we use. In this section, we 
give a more systematic way to define the rings $\Rf$ and ideal classes
\If that works even when $f\equiv 0$. 

Given a base ring \BR, if we form a rank $n$ \BR-module
$R=\BR r_1 \oplus \dots \BR r_n$, we can specify
a $\BR$-bilinear product on $R$ by letting
$$
r_i r_j =\sum_{k=1}^n c_{i,j,k} r_k\quad\textrm{for } c_{i,j,k}\in\BR,
$$
and $e=\sum_{k=1}^n e_{k}r_k$ for some $e_k\in\BR$.  If this
product is commutative, associative, and $e$ is a multiplicative identity 
(which is a queston of certain polynomial equalities with integer coefficients
being satisfied by the $c_{i,j,k}$ and $e_k$) then we call
the $c_{i,j,k}$ and $e_k$ a \emph{multiplication table}.  A multiplication
table gives a ring $R$ with a specified $\BR$-module basis.

Similarly, we can form a free rank $m$ $\BR$-module
$I=\BR \alpha_1 \oplus \dots \BR \alpha_m$, where usually 
 $m$ is a multiple of $n$.  Then we can specify a $\BR$-bilinear 
 product $R\times I \ra \BR$
 by
 $$
r_i \alpha_j =\sum_{k=1}^m d_{i,j,k} \alpha_k\quad\textrm{for } d_{i,j,k}\in\BR.
$$
That this product gives an $R$-module action on $I$
is a question of certain polynomial equalities with integer coefficients
being satisfied by the $d_{i,j,k}$, $c_{i,j,k}$ and $e_k$, and in the case 
they are satisfied we
call the $d_{i,j,k}$  an \emph{action table}.  An action table
gives an $R$-module $I$ with a specified \BR-module basis.

If we want to work directly with forms with $f_0=0$ (for example,
to deal with the form $f\equiv 0$ or to study the form $f=x^2y+xy^2$
when we replace $\Z$ with $\Z/2\Z$), we see that we can define 
a ring $\fRf$ from the multiplication table given 
in Equation~\eqref{E:mult}.  The conditions of commutativity and
associativity on this multiplication table are polynomial
identities in the $f_i$ since the construction of $R$ can also be 
made with the universal form.

 Equations~\eqref{E:Ibasis} and \eqref{E:Isbasis}
 display 
 $\Z$-module bases of $\If$ and $\If^\#$.   
 The action of elements of $\Rf$ on these $\Z$-module bases 
 is given by an action table of
  polynomials in the $f_i$ with $\Z$ coefficients.
  (We can see this, for example, because
  the proofs of Propositions~\ref{P:Ibasis} and \ref{P:Isbasis}
  work over the base ring $\Z[f_0,\dots,f_n]$.)
  These polynomials in the $f_i$ formally give an action table
 because they give an action table over the base ring
  $\Z[f_0,\dots,f_n]$.
    Thus, we can construct 
$\fRf$-modules $\fIf$ and $\fIf^\#$
first as rank $n$ $\Z$-modules and then give them an \fRf 
action by the same polynomials in the $f_i$ that make the action
tables for \If and $\If^\#$ respectively.  

We can also form versions of the powers of \If this way, 
which 
are
$\fRf$-modules that 
we call $\fIf_k$ for $1\leq k \leq n-1$.
We use the action table of $\If^k$
with the basis of Equation~\eqref{E:Ibasis}.
The action table has entries that are integer
polynomials in the $f_i$ for the same reasons as above.
We only have defined
the $\fIf_k$ as \fRf-modules and not as fractional ideals
of \fRf.  Whenever $f\not\equiv0$, however, we have also given a a realization of
the $\fIf_k$ as the ideal class $I_f^k$ (or $I_f^\#$ when $k=-1$).
  Let $\fIf_{-1}:=\fIf^\#$ and $\fIf:=\fIf_1$. 
We do not put the $k$
in the exponent because even when $f$ is non-zero but non-primitive, it is not clear
that the module ${\fIf}_k$ is a power of the module $\fIf$.
When $f$ is primitive, since $I_f$ is invertible, its ideal 
class powers are the same as its module powers.

\subsection{Simple geometric construction}\label{S:ZGeomConst}
For many reasons, we desire a canonical, basis free description
of the ring $\fRf$ and $\fRf$-modules $\fIf_k$.
We would like to deal more uniformly with the case that $f_0=0$ and see easily
the $\GL_2(\Z)$ invariance of our constructions. A binary $n$-ic form $f$ describes a subscheme of
$\P^1_\Z$ which we call $\Sf$.  Let $\mathcal{O}(k)$ denote
the usual sheaf on $\P^1_\Z$ and let
$\mathcal{O}_\Sf(k)$ denote its pullback to \Sf.
 Also, for a sheaf $\mathcal{F}$,
let $\Gamma(U,\mathcal{F})$ be 
sections of $\mathcal{F}$ on $U$ and let
$\Gamma(\mathcal{F})$ be the global sections of $\mathcal{F}$.

\begin{theorem}\label{T:globalbun}
For a binary form $f\not\equiv 0$, the ring $\Gamma(\mathcal{O}_{\Sf})$ of global functions of $\Sf$ is
isomorphic to $\Rf$.
The global sections 
$\Gamma(\mathcal{O}_{\Sf}(k))$ have an $\Gamma(\mathcal{O}_{\Sf})$-module structure,
and since $\Rf\isom\Gamma(\mathcal{O}_{\Sf})$,
 this gives $\Gamma(\mathcal{O}_{\Sf}(k))$ an \Rf-module structure.
For, $1\leq k\leq n-1$, the global sections
$\Gamma(\mathcal{O}_{\Sf}(k))$ are isomorphic
to $\If^k$ as  an $\Rf$-module.
The global sections
$\Gamma(\mathcal{O}_{\Sf}(-1))$ are isomorphic
to $\If^\#$ as  an $\Rf$-module.
\end{theorem}

\begin{proof}
We can act by \GZ so that $f_0\ne 0$ and $f_n \ne 0$. 
Then if we write $\P^1_\Z=\Proj \Z[x,y]$, we can
cover $\P^1_\Z$ with the open subsets
$U_y$ and $U_x$ where $y$ and $x$ are invertible, respectively.
%Let $f(z)$ denote
%$f_0z^n +f_1 z^{n-1} +\dots +f_n$
%and 
%$F(z)$ denote
%$f_0 +f_1 z +\dots +f_nz^n$.
%Then $U_y$ is isomorphic
%to $\Spec \Z[t]/(f(t))$ and
%$U_T$ is isomorpphic
%to $\Spec \Z[T]/(F(T))$.
%The overlap $U_x \cap U_y$ is isomorphic to
%the spectrum of $R_0=\Z[t,T]/(f(t),tT-1,F(T))\isom \Z[t,t^{-1}]/(f(t))$
%with $T$ mapping to $t^{-1}$.
\begin{lemma}
If $f_n\ne 0$, then the restriction map
\begin{align*}
\Gamma(U_y,\OSf(k)) &\ra \Gamma( U_y\cap U_x,\OSf(k))
%\Z[t]/(f(t))&\ra \Z[t,t^{-1}]/(f(t))
\end{align*}
is injective.
\end{lemma}
\begin{proof}
If $\sum_{i\geq -k} a_ix^{k+i}y^{-i} \mapsto 0$, with $a_i\in\Z$, then
$\sum_{i\geq -k} a_ix^{k+i}y^{-i}=\sum_j d_j x^jy^{k-n-j}f$, where
$d_j \in \Z$.  
Since
$\sum_{i\geq -k} a_ix^{k+i}y^{-i}$ has no terms of negative
degree in $x$ and $f_n\ne0$, we conclude that
$d_j=0$ for $j<0.$  Thus, $\sum_{i\geq -k} a_ix^{k+i}y^{-i}$
is 0 in $\Gamma( U_y,\OSf(k))$.
\end{proof}

Similarly, since $f_0\ne 0$, we have that
$\Gamma (U_x,\OSf(k)) \ra \Gamma( U_y\cap U_x,\OSf(k))$
is an injection.  

So we wish to determine the elements of
$\Gamma( U_y\cap U_x,\OSf(k))$
that are in the images
of both $\Gamma( U_x,\OSf(k))$ and $\Gamma( U_y,\OSf(k))$
First, note that $x^k, x^{k-1}y,\dots, y^k$
are in the images of both restriction maps.
In 
$\Gamma( U_y\cap U_x,\OSf(k))$ for $1\leq m \leq n-k-1$, we have
$$f_0 x^{k+m}y^{-m} + \dots + f_{k+m-1}xy^{k-1} = -f_{k+m}y^{k} -\dots - f_n x^{k+m-n}y^{n-m},$$
and thus
$z_m:=f_0 x^{k+m}y^{-m} + \dots f_{k+m-1}xy^{k-1}$ is 
in the images of both
$\Gamma( U_x,\OSf(k))$ and $\Gamma( U_y,\OSf(k))$.

Now, let $p$ be in both images so that
 $p=\sum_{i\geq -k} a_i x^{k+i}y^{-i}=\sum_{i\leq -k} b_i x^{-i}y^{k+i}$
 with $a_i,b_i\in\Z$.  If $a=\sum_{i\geq -k} a_i x^{k+i}y^{-i}\in \Gamma( U_y,\OSf(k))$ and 
 $b=\sum_{i\leq 0} b_i x^{-i}y^{k+i}\in \Gamma( U_x,\OSf(k))$, then we have a formal equality 
 $a-b=\sum_{i} c_i x^iy^{k-i-n} f$ 
 (in $\Z[x,x^{-1},y,y^{-1}]$)
 where $c_i\in \Z$.
 We can assume without loss of generality
 that $c_i=0$ for $i\geq 0$ because
 any $c_i x^iy^{k-i-n} f$ with $i$ non-negative
 we could just subtract from the representation $a$
 to obtain another such representation of $p$ in $\Gamma( U_y,\OSf(k))$.
 Similarly, we can assume that $c_i=0$ for $i\leq k-n$.
 From the equality 
 $a-b=\sum_{i=-n+k+1}^{-1} x^iy^{k-i-n} f$, we can conclude that
 $a$ is a linear combination of $x^k, x^{k-1}y,\dots, y^k$
  plus all the terms 
  $\sum_{i=-n+k+1}^{-1} x^iy^{k-i-n} f$
  of positive degree in $x$,
 and $b$ is that same linear combination minus all the 
 terms of $\sum_{i=-n+k+1}^{-1} x^iy^{k-i-n} f$ of
 positive degree in $x$.
 The terms of positive degree in $x$ of $x^iy^{k-i-n} f$ sum to
 $z_{n+i-k}$.
 Thus, $a\in\ang{x^k, x^{k-1}y,\dots, y^k,z_1,\dots,z_{n-1-k}}_\Z$.  

For $k\geq 0$, 
when we map $\ang{x^k, x^{k-1}y,\dots, y^k,z_1,\dots,z_{n-1-k}}_\Z$ to
 $\Qf$ via $x \mapsto \theta$ and $y \mapsto 1$, the image
 is the free rank $n$ \Z-module 
 $\ang{1,\theta, \dots, \theta^k, \zeta_{k+1},\dots,\zeta_{n-1}}_\Z$.
 Thus, the map is an isomorphism of $\ang{x^k, x^{k-1}y,\dots, y^k,z_1,\dots,z_{n-1-k}}_\Z$,
   the global sections of $\OSf(k)$, onto $\If^k$.
   Clearly the $\Gamma(\OSf)$-module structure on
   $\Gamma(\OSf(k))$ is the same as the 
   the $\Rf$-module structure on $\If^k$ (including the $k=0$ case, which gives
the ring isomorphism $\Rf\isom\Gamma(\OSf)$).
 When $k=-1$,
 when we map $\ang{z_1,\dots,z_{n}}_\Z$ to
 $\Qf$ via $x \mapsto \theta$ and $y \mapsto 1$, the image
 is the free rank $n$ \Z-module 
 $\If^\#$. Similarly we conclude the theorem for $\If^\#$.
\end{proof}

Note that 
though $\mathcal{O}_\Sf(k)$ is always an invertible $\mathcal{O}_{\Sf}$-module,
when $\Sf$ is not affine,
the global sections $\Gamma(\mathcal{O}_{\Sf}(k))$ are not necessarily
an invertible $\Gamma(\mathcal{O}_{\Sf})$-module.  In fact, we know that
for nonzero $f$ and  $1\leq k \leq n-1$ that
$\Gamma(\mathcal{O}_\Sf(k))$ is an invertible $\Gamma(\mathcal{O}_\Sf)$-module
exactly when $f$ is primitive.

\begin{theorem}\label{T:affine}
Let $f$ be a binary form with non-zero discriminant.
The scheme \Sf is affine if and only if $f$ is primitive.
\end{theorem}
\begin{proof}
From Theorem~\ref{T:globalbun}
we see that if \Sf is affine, then since
$\Gamma(\mathcal{O}_{\Sf}(1))\isom \If$ and $\mathcal{O}_{\Sf}(1)$ is  invertible we must
have that $ \If$ is an invertible $\Rf$-module.  Thus 
by Proposition~\ref{P:Iinvert}, if $\Sf$ is affine then
$f$ is primitive.  We see that
 $\Sf$ has a vertical fiber over $(p)$ when $p\mid f$.
 Moreover, when $p\mid f$ we see from
 the multiplication table (Equation~\eqref{E:mult}) that the fiber of 
 $a$ over $(p)$ is the non-reduced $n$-dimensional point
 $\Spec \Z_{/(p)} [x_1,x_2,\dots,x_{n-1}]/(x_ix_j)_{1\leq i,j\leq n-1}$
which does not embed into $\P^1_\Z$.  

Now suppose that $f$ is primitive and has non-zero discriminant.
We can change variables so that $f_0\ne 0$ and $f_n\ne 0$.
From the standard open affine cover of $\P^1_\Z$, we have that
$\Sf$ is covered by affine opens $U_y=\Spec \Z [x/y]/(f/y^n)$ and
$U_x=\Spec \Z [y/x]/(f/x^n)$.  Since $\Rf$ is a finitely generated
$\Z$-module inside $\Qf$ (which is a product of number fields), we know
that the class group of $\Rf$ is finite.  So, let $m$ be such that
$(\If^\#)^m$ is principal.  (Note that by Proposition~\ref{P:invertextra} we
know that $(\If^\#)$ is an invertible $\Rf$-module.)  Let
$J=\theta \If^\#$ which is an integral $\Rf$-ideal.  
Let $J^m=(b)$ and $(\If^\#)^m=(a)$, with $a,b \in \Rf$.  
As in the computation in the proof of
Proposition~\ref{P:invertextra}, we see that $\If^\# + J =(1)$
and thus there exists $\alpha,\beta\in\Rf$ such that $\alpha a+\beta b=1$.  We claim
that $(\Sf)_{a}=U_y$ as open subschemes of $\Sf$, where
$(\Sf)_{a}$ denotes the points of $\Sf$ at which $a$ is non-zero.  

In the ring $\Qf$ we have that $a \theta^m =b u$, where $u$ is a unit
in $\Rf$. 
 In $\Gamma(U_y\cap U_x,\OSf)\isom \Qf$ this translates to $a (\frac{x}{y})^m =b u$
 Thus 
$$a(\alpha+\frac{\beta}{u}\left(\frac{x}{y}\right)^m)= \alpha a+\beta\frac{a}{u}\left(\frac{x}{y}\right)^m=\alpha a+\beta b=1$$
in $\Gamma(U_y,\OSf)$ (which injects into $\Gamma(U_y\cap U_x,\OSf)\isom \Qf$).
Therefore $a$ is not zero at any point of $U_y$, and so $U_y \subset (\Sf)_{a}$.
  Suppose that  we have a point $p\not\in U_y$ so that
$\frac{y}{x}$ is 0 at $p$.  
Since in $\Gamma(U_y\cap U_x,\OSf)$ we have $a  =b u (\frac{y}{x})^m$, this is also true
in $\Gamma(U_x,\OSf)\isom \Z [y/x]/F(y/x)$ (which injects into $\Gamma(U_y\cap U_x,\OSf)$).
Since we have $p\in U_x$, then $a$ is also 0 at $p$ and so 
$p\not\in(\Sf)_{a}$ and we conclude $(\Sf)_{a}\subset U_y$.
We have shown $(\Sf)_{a}=U_y$ and by switching $x$ and $y$
we see similarly that $(\Sf)_{b}=U_x$.  Since $(a,b)$ is the
unit ideal in $\Gamma(\Sf,\OSf)\isom R_f$, and $(\Sf)_{a}$
and $(\Sf)_{b}$ are each affine, we have that $\Sf$ is affine
(\cite[Exercise 2.17(b)]{Hartshorne}).

We could similarly argue over a localization of $\Z$, and thus
 localizing away from the $\Z$ primes that divide $f$, the scheme
$\Sf$ is the same as $\Spec R_f$.  Over the primes of $\Z$
that divide $f$, $S_f$ has vertical fibers isomorphic to $\P^1_{\Z/p\Z}$
 but
$\Spec R_f$ has a non-reduced $n$-dimensional point.
\end{proof}

\subsection{Geometric construction by hypercohomology}\label{SS:hyperco}

The description of $\Rf$ as the global functions of the subscheme given
by $f$ is very satisfying as a coordinate-free, canonical and simple 
description of $\Rf$, but still does not take care of the form $f\equiv 0$.  
It may seem at first that $f\equiv 0$ is a pesky, uninteresting case,
but we will eventually want to reduce a form so that
its coefficients are in $\Z/p\Z$, in which
case many of our non-zero forms will go to 0.  In general we may want to 
base change, and the formation of the ring $\Gamma(\mathcal{O}_{\Sf})$
does not commute with base change.  For example,
a non-zero binary $n$-ic all of whose coefficients are divisible $p$ will give
a rank $n$ ring $\Gamma(\mathcal{O}_{\Sf})$ but the reduction $\bar{f}$ of $f$ to $\Z/p\Z$ would 
give $S_{\bar{f}}=\P^1_{\Z/p\Z}$ and thus a ring of global functions that is rank 1
over $\Z/p\Z$.

We can, however, make the following construction, which was given
for $n=3$ by Deligne in a letter \cite{Delcubic} to Gan, Gross, and Savin.
On $\P^1_\Z$ a binary $n$-ic form $f$ gives $\mathcal{O}(-n) \stackrel{f}{\ra} \mathcal{O}$,
whose image is the ideal sheaf of $\Sf$.  We can consider 
$\mathcal{O}(-n) \stackrel{f}{\ra} \mathcal{O}$ as a complex in degrees -1 and 0, and then take
the hypercohomology of this complex:
\begin{equation}
R=H^0R\pi_* \left(\mathcal{O}(-n) \stackrel{f}{\ra} \mathcal{O}\right).
\end{equation}
(Here we are taking the 0th right hyper-derived functor
of the pushforward by $\pi : \P^1_\Z \ra \Spec\Z$
on this complex.  Alternatively, we pushforward the complex in the derived category and then take $H^0$.
We take hypercohomology since we are applying
the functor to a complex of sheaves and not just a single sheaf.)
There is a product on the complex $\mathcal{O}(-n) \stackrel{f}{\ra} \mathcal{O}$ given
as $\mathcal{O} \tensor \mathcal{O} \ra \mathcal{O}$ by multiplication, $\mathcal{O}\tensor \mathcal{O}(-n)\ra\mathcal{O}(-n)$ by the
$\mathcal{O}$-module action, and $\mathcal{O}(-n)\tensor\mathcal{O}(-n)\ra 0$. 
This product is clearly commutative and associative, and induces a product on $R$.
 The map of complexes
$$
\begin{CD}
 @. \mathcal{O} @.\\
@. @VVV @.\\
\mathcal{O}(-n) @>>> \mathcal{O} @.
\end{CD} 
$$
induces $\Z\ra  R$.  
(Of course, 
$H^0R\pi_* (\mathcal{O})$ is just $\pi_*(\mathcal{O})\isom \Z$.)
  It is easy to see that
$1\in H^0R\pi_* (\mathcal{O})$ acts as the multiplicative identity.

When $f\not\equiv0$, the map $\mathcal{O}(-n) \stackrel{f}{\ra} \mathcal{O}$ is injective, and
thus the complex $\mathcal{O}(-n) \stackrel{f}{\ra} \mathcal{O}$ is 
 homotopy equivalent to $\mathcal{O} / f(\mathcal{O}(-n))\isom \mathcal{O}_\Sf$ (as a chain complex in the 0th degree).
The homotopy equivalence also respects the product structure on the complexes.
Thus when $f\not\equiv0$, we have
$R\isom\pi_*(\mathcal{O}_{\Sf})$, and since $\Spec \Z$ is affine
we can consider $\pi_*(\mathcal{O}_{\Sf})$ simply as a $\Z$-module isomorphic
to $\Gamma(\mathcal{O}_{\Sf})\isom R_f$.  When $f\equiv 0$ we have 
$$
R=H^0R\pi_*(\mathcal{O}) \oplus H^1R\pi_*(\mathcal{O}(-n)) \isom \Z \oplus \Z^{n-1}
$$
as a $\Z$-module and with multiplication 
given by $(1,0)$ acting as the multiplicative identity and
$(0,x)(0,y)=0$ for all $x,y \in \Z^{n-1}$.
 This agrees with
the definition of $\mathcal{R}_0$ given in Section~\ref{S:tables} that
used the coefficients of $f$ to give a multiplication table for \fRf.
So we see this definition of $R$ is a natural extension to all $f$ of the construction
$\Gamma(\OSf)$ for non-zero $f$, especially since $R$ gives a rank $n$
ring even when $f\equiv 0$.  

\begin{theorem}\label{T:twocon}
For all binary $n$-ic forms $f$, we have
$$\fRf\isom H^0R\pi_* \left(\mathcal{O}(-n) \stackrel{f}{\ra} \mathcal{O}\right)$$ as rings.
(Note that $\fRf$ is defined in Section~\ref{S:tables}.)
\end{theorem}
\begin{proof}
The proof of Theorem~\ref{T:globalbun} 
shows that $\fRf\isom H^0R\pi_* \left(\mathcal{O}(-n) \stackrel{f}{\ra} \mathcal{O}\right)$
for the universal form $f=f_0x^n +f_1 x^{n-1}y +\dots +f_ny^n$ with 
coefficients in $\Z[f_0,\dots,f_n]$.  Since both the construction of
\fRf from the multiplication table in Section~\ref{S:tables} and
the formation of $ H^0R\pi_* \left(\mathcal{O}(-n) \stackrel{f}{\ra} \mathcal{O}\right)$
commute with base change (as we will see in Theorem~\ref{T:basechpf}), and every
form $f$ is a base change of the universal form, the theorem follows.
\end{proof}

We have a similar description of the
$\fRf$ ideal classes (or modules) $\fIf_k$.
We can define $\fRf$-modules for 
all $k\in\Z$: 
$$
 H^0R\pi_* \left(\mathcal{O}(-n+k) \stackrel{f}{\ra} \mathcal{O}(k)\right).
$$
  (Here, $\mathcal{O}(k)$ is in degree 0 in the above complex.)
  The $\fRf$-module structure on
$$H^0R\pi_* \left(\mathcal{O}(-n+k) \stackrel{f}{\ra} \mathcal{O}(k)\right)$$ is given by the following
action of the complex $\mathcal{O}(-n) \stackrel{f}{\ra} \mathcal{O}$
on the complex 
 $\mathcal{O}(-n+k) \stackrel{f}{\ra} \mathcal{O}(k):$
 \begin{align*}
 \mathcal{O}\tensor \mathcal{O}(k) \ra \mathcal{O}(k) \qquad \mathcal{O}\tensor \mathcal{O}(-n+k) \ra \mathcal{O}(-n+k)\\
 \mathcal{O}(-n)\tensor \mathcal{O}(k) \ra \mathcal{O}(-n+k) \qquad \mathcal{O}(-n)\tensor \mathcal{O}(-n+k) \ra 0,
 \end{align*}
 where all maps are the natural ones.

\begin{theorem}\label{T:Icon}
For all binary $n$-ic forms $f$ and $-1\leq k\leq n-1$ 
we have
$$
\fIf_k\isom H^0R\pi_* \left(\mathcal{O}(-n+k) \stackrel{f}{\ra} \mathcal{O}(k)\right)
$$
as $\fRf$-modules.  
\end{theorem}
\begin{proof}
The proof is that same as that of Theorem~\ref{T:twocon}.
\end{proof}

We have the following nice corollary of Theorems~\ref{T:twocon}
and \ref{T:Icon}.

\begin{corollary}\label{C:fGinv}
The ring $\fRf$ and the $\fRf$-module $\fIf$ are $\GZ$
invariants of binary $n$-ic forms $f$.
\end{corollary} 

\section{Constructing rings and modules from a binary form over an arbitrary base}\label{S:SConst}

So far, we have mainly considered binary  forms with coefficients
in $\Z$.  We will now develop our theory over an arbitrary base scheme
$\BS$.  When $\BS=\Spec \BR$ we will sometimes say we are working over a base ring
$\BR$ and we will replace $\mathcal{O}_\BS$-modules with
their corresponding $\BR$-modules. 

\begin{notation}
For an $\OS$-module $M$, we write $M^*$ to denote the $\OS$ dual
$\mathcal{H}om_\OS(M,\OS)$. If $\mathcal{F}$ is a sheaf, we use $s\in\mathcal{F}$ to denote that $s$ is a global section of
$\mathcal{F}$.   We use $\Sym^n M$ to denote the usual
quotient of $M^{\tensor n}$, and $\Sym_n M$ to denote the submodule of symmetric elements of $M^{\tensor n}$. 
We have $(\Sym_n M)^* \cong \Sym^n M^*$ for locally free $\OS$-modules $M$.
\end{notation}

 A \emph{binary $n$-ic form}
over
$\BS$ is a pair $(f,V)$ where
$V$ is a locally free $\mathcal{O}_\BS$-module of rank 2 and
$f\in\Sym^n V$. An isomorphism of
binary $n$-ic forms $(f,V)$ and $(f,V')$ is given 
by an $\mathcal{O}_\BS$-module isomorphism $V\isom V'$ which takes
$f$ to $f'$.
We call $f$ a \emph{binary form} when $n$ is clear from context or not
relevant.  If $V$ is the free
$\mathcal{O}_\BS$-module $\mathcal{O}_\BS x \oplus \mathcal{O}_\BS y$ we call $f$
a
\emph{free} binary form.

Given a binary form $f\in \Sym^n V$ over a base scheme $\BS$,
the form $f$ determines a subscheme $\Sf$ of $\P(V)$
(where we define $\P(V)=\Proj \Sym^* V$).
Let $\pi: \P(V) \ra S$.  Let $\mathcal{O}(k)$ denote
the usual sheaf on $\P(V)$ and $\OSf(k)$ denote the pullback
of $\mathcal{O}(k)$ to
\Sf. Then we can define the $\mathcal{O}_\BS$-algebra
\begin{equation}
\fRf:=H^0R\pi_* \left(\mathcal{O}(-n) \stackrel{f}{\ra} \mathcal{O}\right),
\end{equation}
where $\mathcal{O}(-n) \stackrel{f}{\ra} \mathcal{O}$ is a complex in degrees -1 and 0.
(In section~\ref{SS:hyperco} this point of view  is worked out
in detail over $\BS=\Spec \Z$.)  The product of $\fRf$ is 
given
by the natural product of the complex $\mathcal{O}(-n) \stackrel{f}{\ra} \mathcal{O}$
with itself and the $\mathcal{O}_\BS$-algebra structure is induced from the map of $\mathcal{O}$
as a complex in degree 0 to the complex $\mathcal{O}(-n) \stackrel{f}{\ra} \mathcal{O}$.

When $\mathcal{O}(-n) \stackrel{f}{\ra} \mathcal{O}$ is injective,  we have $$\fRf=\pi_*(\OSf),$$
as in Section~\ref{SS:hyperco}.
Similarly, we can define an $\fRf$-module 
\begin{equation}
\fIf_k:=H^0R\pi_* \left(\mathcal{O}(-n+k) \stackrel{f}{\ra} \mathcal{O}(k)\right),
\end{equation}
for all $k\in \Z$. 
 Let $\fIf^\#:=\fIf_{-1}$ and $\fIf:=\fIf_1$.
 Clearly $\fRf$ and $\fIf_k$ are invariant
under the $\GL(V)$ action on forms in
$\Sym^n V$.  
Again, when $\mathcal{O}(-n+k) \stackrel{f}{\ra} \mathcal{O}(k)$ is injective,  we have $$\fIf_k=\pi_*(\OSf(k))$$ for all $k\in\Z$.

\begin{example}
If $\BR=\Z\oplus \Z$ and $(f_i)=\Z\oplus\{0\}$, then 
in $\P^1_{\Z\oplus \Z}$
over the
first $\Spec \Z$ the form $f$ cuts out $\Spec R_{p(f)}$, where
$p(f)$ is the projection of $f$ onto the first factor
of $(\Z\oplus \Z) [x,y]$.  
Over the second copy of $\Spec \Z$, the form $f$ is 0 and
cuts out all of $\P^1_\Z$.  Here
$\mathcal{O}(-n)\stackrel{f}{\ra}\mathcal{O}$ is not injective because $f$ is a 0 divisor.
Thus the ring $\fRf:=H^0R\pi_* (\mathcal{O}(-n) \stackrel{f}{\ra} \mathcal{O})$ is not just the global functions of
$\Sf$ but also has a contribution from $\ker (\mathcal{O}(-n)\stackrel{f}{\ra}\mathcal{O} )$.
\end{example}

Unlike pushing forward $\OSf(k)$ to $S$, the constructions of $\fRf$ and $\fIf_k$ for $-1\leq k\leq n-1$
commutes with base change.

\begin{theorem}\label{T:basechpf}
Let $f\in \Sym^n V$ be a binary form over a base scheme $\BS$.  
The construction of $\fRf$ and $\fIf_k$ for $-1\leq k\leq n-1$ commutes with base change.
More precisely, let
$\phi : T \ra \BS$ be a map of schemes.
Let $\phi^*f\in \Sym^n \phi^*V$ be the pullback of $f$.  Then the natural
map from cohomology
$$\fRf \tensor \mathcal{O}_T \ra \fR{\phi^*f}$$ is an isomorphism of $\mathcal{O}_T$-algebras.
Also, for $-1\leq k\leq n-1$, the 
natural
map from cohomology
$$\fIf \tensor \mathcal{O}_T \ra \fI{\phi^*f}$$ is an isomorphism of $\fR{\phi^*f}$-modules
(where the $\fR{\phi^*f}$-module structure on $\fIf \tensor \mathcal{O}_T$ comes from
the $(\fRf \tensor \mathcal{O}_T)$-module structure.
\end{theorem}
\begin{proof}
The key to this proof is to compute all cohomology of the pushforward 
of the complex $\mathcal{C}(k): \mathcal{O}(-n+k) \stackrel{f}{\ra} \mathcal{O}(k)$.
This can be done using the long exact sequence of cohomology from the short exact sequence
of complexes given in Equation~\eqref{E:SEScomplexes} in the next section.  
In particular, $\mathcal{C}(k)$ 
 does not have any cohomology
in degrees other than 0.  
Since $k\leq n-1$, we have that 
$H^{0}R\pi_*(\mathcal{O}(-n+k))=0$
and thus
$H^{-1}R\pi_*(\mathcal{C}(k))=0$.
Since, $k\geq -1$ we have that $H^{1}R\pi_*(\mathcal{O}(k))=0$ and thus
$H^{1}R\pi_*(\mathcal{C}(k))=0$.  Moreover, in Section~\ref{SS:LES}, we see that
$H^{0}R\pi_*(\mathcal{C}(k))$ is locally free.  Thus since all $H^{i}R\pi_*(\mathcal{C}(k))$ are flat,
by \cite[Corollaire 6.9.9]{EGA3}, we have that cohomology and base change commute.
\end{proof}

In the case that $f$ is a free form, we could have defined
$\fRf$ as a free rank $n$ $\OS$-module using the multiplication table given
by Equation~\eqref{E:mult} and $\fIf_k$ for $-1\leq k \leq n-1$ as a free rank $n$ $\OS$-module
using the action tables for the 
Equation~\eqref{E:Ibasis} and \eqref{E:Isbasis} bases.
(See Section~\ref{S:tables} for more details.)
Both the constructions from hypercohomology described above and
from the multiplication and action tables commute with base change.
Thus by verification on the universal form (the proof
of Theorem~\ref{T:globalbun}
works over 
$\Z[f_0,\dots,f_n]$) we see, as in Theorem~\ref{T:twocon}, that
for free binary forms and $-1\leq k \leq n-1$,
these two definitions of $\fRf$ and $\fIf_k$ agree.

For any $l$, we can also formulate this theory for 
\emph{$l$-twisted binary forms}  $f\in\Sym^n V\tensor (\wedge^2 V)^{\tensor l}$,
where
\begin{equation}
\fRf:=H^0R\pi_* \left(\mathcal{O}(-n)\tensor (\pi^*\wedge^2 V)^{\tensor -l} \stackrel{f}{\ra} \mathcal{O}\right),
\end{equation}
and 
\begin{equation}
\fIf_k:=H^0R\pi_* \left(\mathcal{O}(-n+k)\tensor (\pi^*\wedge^2 V)^{\tensor -l} \stackrel{f}{\ra} \mathcal{O}(k)\right)\,
\end{equation}
or
\begin{equation}
\fIf_k':=H^0R\pi_* \left(\mathcal{O}(-n+k)\tensor (\pi^*\wedge^2 V) \stackrel{f}{\ra} \mathcal{O}(k)\tensor (\pi^*\wedge^2 V)^{\tensor l+1}\right).
\end{equation}
By the projection formula, $\fIf_k'=\fIf_k\tesnor (\wedge^2 V)^{\tensor l+1}$.
By an argument analogous to that of Theorem~\ref{T:basechpf} we find that these constructions also
commute with base change for $-1\leq k \leq n-1$.  
Note that since $\Sym ^n V \tensor (\wedge^2 V)^{\tensor l}\isom \Sym ^n V^* \tensor (\wedge^2 V^*)^{\tensor -n-l}$
(see Lemmas~\ref{P:symline} and \ref{L:rk2} in the Appendix),
the theory of $l$-twisted binary $n$-ic forms is equivalent to the theory of $(-n-l)$-twisted binary $n$-ic forms.

\subsection{Long exact sequence of cohomology}\label{SS:LES}
In this section, we use the long exact sequence of cohomology to find the $\OS$-module structures of the rings and modules we have constructed, and
important relationships between these $\OS$-module structures.
From the short exact sequence of complexes in degrees
-1 and 0
\begin{equation}\label{E:SEScomplexes}
\begin{CD}
 @. \mathcal{O}(k) \\
@. @VVV \\
\mathcal{O}(-n+k) @>f>> \mathcal{O}(k) \\
@VVV @.  \\
\mathcal{O}(-n+k) @. 
\end{CD} 
\end{equation}
(where each complex is on a horizontal line),
we have the long exact sequence of cohomology
\begin{align*}
H^0R\pi_*\mathcal{O}(-n+k)\ra H^0R\pi_*\mathcal{O}(k) \ra 
H^0R\pi_* &\left(\mathcal{O}(-n+k) \stackrel{f}{\ra}  \mathcal{O}(k)\right)  \\
&\ra H^1R\pi_*\mathcal{O}(-n+k)
\ra H^1R\pi_*\mathcal{O}(k).
\end{align*}
For $k\leq n-1$, we have $H^0R\pi_*\mathcal{O}(-n+k)=0$ and for $k\geq -1$
we have $H^1R\pi_*\mathcal{O}(k)=0$. Also,
$H^0R\pi_*\mathcal{O}(k)=\Sym^k V$ and $H^1R\pi_*\mathcal{O}(-n+k)=(\Sym^{n-k-2} V)^* \tesnor (\wedge^2 V)^*.$
Thus for $1\leq k\leq n-1$
and a binary form $f\in\Sym^n V$, we have the exact sequence
\begin{equation}\label{E:Iexact}
0\ra \Sym^k V \ra 
\fIf_k \ra (\Sym^{n-k-2} V)^* \tesnor (\wedge^2 V)^*
\ra 0.
\end{equation}
Thus $\fIf_k$ has
a
canonical rank $k+1$ $\mathcal{O}_\BS$-module inside of it
(coming from the global sections $x^k,x^{k-1}y,\dots,y^k$ of $\mathcal{O}(k)$), and
a canonical rank $n-k-1$ $\mathcal{O}_\BS$-module quotient.

So we see, for example, that as an $\mathcal{O}_\BS$-module
$$
\fRf/\mathcal{O}_\BS \isom (\Sym^{n-2} V)^* \tesnor (\wedge^2 V)^*.
$$

Note that if we make the corresponding exact sequence for an 
$l$-twisted binary form  $f\in\Sym^n V\tensor (\wedge^2 V)^{\tensor l}$
we have
\begin{equation}\label{E:Iexactl}
0\ra \Sym^k V \ra 
\fIf_k \ra (\Sym^{n-k-2} V)^* \tesnor (\wedge^2 V)^{\tensor -l-1}
\ra 0
\end{equation}
or
\begin{equation}\label{E:Iexactlprime}
0\ra \Sym^k V \tesnor (\wedge^2 V)^{\tensor l+1} \ra 
\fIf_k' \ra (\Sym^{n-k-2} V)^* 
\ra 0.
\end{equation}

In Section~\ref{S:RI} we have given a multiplication table
for an explicit basis of $\fRf$ and an (implicit) action table
for an explicit basis of $\fIf_k$.  One naturally wonders how those
bases relate to the exact sequences that we have just found.  
Consider the universal form $f$ over the base ring $B=\Z[f_0,\dots,f_n]$.  
We can use the concrete construction of $\fRf$ and $\fIf_{k}$ in
Section~\ref{S:RI}.  If $K$ is the fraction field of $B$, then the concrete constructions of 
$\fRf$ and $\fIf_{k}$ lie in $\Qf:=K(\theta)/(f_0\theta^n +f_1 \theta^{n-1} +\dots +f_n)$
and are given by Equations~\eqref{E:basis} and \eqref{E:Ibasis}.

\begin{proposition}\label{P:identifybasis}
 For the universal form $f$, where $V$ is a free module on $x$ and $y$,
in the exact sequence of Equation~\eqref{E:Iexactl} or Equation~\eqref{E:Iexactlprime} 
(with $\wedge^2 V$ trivialized by the basis element $x\wedge y$)
we have that 
$$x^iy^{k-i} \in \Sym^k V \text{ is identified with } \theta^i\in \fIf_k\text{ for }0\leq i\leq k$$
and
$$
\text{the dual basis to }
x^{n-k-i-1}y^{i-1} \in \Sym^{n-k-2} V \text{ is identified with } $$
$$\zeta_{k+i} \in \fIf_k\text{ for }1\leq i\leq n-k-1.$$
\end{proposition}

\begin{proof}
For the universal form, the cohomological construction simplifies.  We can
replace the complex $\mathcal{O}(-n+k)\ra\mathcal{O}(k)$ on $\P^1_B$ with the single
sheaf $\mathcal{O}(k)/f(\mathcal{O}(-n+k))$.  We can then replace $H^i R\pi_*$ with $H^i$
since the base is affine.  The short exact sequence of complexes
in Equation~\eqref{E:SEScomplexes} then simplifies to the short
exact sequence of sheaves
$$
0\ra \mathcal{O}(-n+k) \stackrel{f}{\ra} \mathcal{O}(k) \ra \mathcal{O}(k)/f(\mathcal{O}(-n+k)) \ra 0,
$$
which gives the same long exact sequence leading to Equation~\eqref{E:Iexact}.
The identification of $\fIf_k$ with global sections is at 
the end of proof of Theorem~\ref{T:globalbun}, and from that it is easy to see that
 the map $H^0(\P^1_B,\mathcal{O}(k))\ra H^0(\P^1_B,\mathcal{O}(k)/f(\mathcal{O}(-n+k)))=\fIf_k$
sends $x^iy^{k-i}\mapsto \theta^i$.  To compute the $\delta$ map 
$\fIf_k\ra H^1(\P^1_B,\mathcal{O}(-n+k))$, we will next use Cech cohomology for
the usual affine cover of $\P^1$ and the $\delta$ map is the snake lemma map
between rows of the Cech complexes.

In the notation of Theorem~\ref{T:globalbun}, 
the element $\zeta_{k+i}$ is identified with the global section $z_i$.
The global function 
$z_i$ pulls back to $z_i\in \Gamma(U_x,  \mathcal{O}(k)) \times\Gamma(U_y,\mathcal{O}(k))$
which maps to $f/(x^{n-k-i}y^{i})\in \Gamma(U_x \cap U_y,  \mathcal{O}(k))$.
This pulls back to $1/(x^{n-k-i}y^{i})\in \Gamma(U_x \cap U_y,  \mathcal{O}(-n+k))$,
which in the standard pairing of the cohomology of projective space
(e.g. in \cite[III, Theorem 5.1]{Hartshorne})
pairs with $x^{n-k-i-1}y^{i-1}\in H^0(\P^1_B,\mathcal{O}(n-k-2))\isom \Sym_{n-k-2} V$.
\end{proof}

Since the ring $\fRf$ acts on $\fIf_k$, it is natural to want to understand this action
in terms of the exact sequences of Equation ~\eqref{E:Iexactl}.  We have the following
description, which can be proved purely formally by the cohomological constructions of everything
involved.  Alternatively, with the concrete description of the basis elements
in Proposition~\ref{P:identifybasis}, one could prove the following by computation.

\begin{proposition}\label{P:rightaction}
 The map $\fRf/\OS\tensor \Sym^k V \ra \Sym_{n-k-2} V^*\tesnor (\wedge^2 V)^{\tensor -l-1}$
given by the action of $\fRf$ on $\fIf_k$ and the exact sequence of Equation ~\eqref{E:Iexactl}
is identified with the natural map (see Lemma~\ref{L:Syminside} in the Appendix)
$$ \Sym_{n-2} V^* \tesnor (\wedge^2 V)^{\tensor -l-1}\tensor \Sym^k V \ra \Sym_{n-k-2} V^*\tesnor (\wedge^2 V)^{\tensor -l-1}$$
under the identification $R/\OS \isom \Sym_{n-2} V^*\tesnor (\wedge^2 V)^{\tensor -l-1}$ of
Equation~\eqref{E:Iexactl}.
 The map $\fRf/\OS\tensor \Sym^k V \tesnor (\wedge^2 V)^{\tensor l+1}\ra \Sym_{n-k-2} V^*$
given by the action of $\fRf$ on $\fIf_k'$ and the exact sequence of Equation ~\eqref{E:Iexactlprime}
is identified with the natural map (see Lemma~\ref{L:Syminside} in the Appendix)
$$ \Sym_{n-2} V^* \tesnor (\wedge^2 V)^{\tensor -l-1}\tensor \Sym^k V \tesnor (\wedge^2 V)^{\tensor l+1}\ra \Sym_{n-k-2} V^*$$
under the identification $R/\OS \isom \Sym_{n-2} V^*\tesnor (\wedge^2 V)^{\tensor -l-1}$ of
Equation~\eqref{E:Iexactl}.

\end{proposition}

\subsection{Dual modules}

For $-1\leq k\leq n-1$ we have a map
\begin{equation}\label{E:dualmap}
 \fIf'_k \tensor \fIf_{n-2-k} \ra \fIf'_{n-2}\ra \OS.
\end{equation}
The first map is induced from the map from the product of the complexes used to define
$\fIf'_k$ and $\fIf_{n-2-k}$ to the complex used to define $\fIf'_{n-2}$.  The second map comes from Equation~\eqref{E:Iexactlprime}.  

\begin{theorem}
 The pairing in Equation~\eqref{E:dualmap} gives an $\OS$-module map
$$
\fIf_k'\ra 
 \fIf^*_{n-2-k},
$$
and this map is an $\fRf$-module isomorphism.
\end{theorem}
\begin{proof}
 We will show that this map is an $\fRf$-module isomorphism by checking on the universal form.
Since all forms are locally a pull-back from the universal form and these constructions commute with base change, the
theorem will follow for all forms.

We use the construction of $\fRf$ and $\fIf'_k$ and $\fIf_{n-2-k}$ in Section~\ref{S:RI}.  (Note that
for the universal form, we trivialize all $\wt V$ with the basis $x\wedge y$ and so $\fIf'_k=\fIf_k$.)
Since the complex used to define  $\fIf_i$ is chain homotopy equivalent to 
the sheaf $\mathcal{O}(i)/f(\mathcal{O}(i-n))$, we see that the map $ \fIf'_k \tensor \fIf_{n-2-k} \ra \fIf'_{n-2}$
is just the multiplication of global sections of $\mathcal{O}(k)_\Sf$ and $\mathcal{O}(n-2-k)_\Sf$ to obtain a global section of
$\mathcal{O}(n-2)_\Sf$.  This can be realized by multiplication of elements of the fractional ideals 
$\fIf_k$, $\fIf_{n-2-k}$, and $\fIf_{n-2}$ in Section~\ref{S:RI}.

\begin{lemma}
 Consider the $\OS$-module basis $$1,\theta,\dots,\theta^k,\zeta_{k+1}+f_{k+1},\dots,\zeta_{n-1}+f_{n-1}$$ for
 $\fIf_k'$.
For $\fIf_{n-2-k}$, consider the $\OS$-module basis of Equation~\eqref{E:basis2}, but reverse the order to obtain
$$f_0\theta^{n-1}+f_1\theta^{n-2}+\dots+f_{k}\theta^{n-k-1}, \dots, f_0\theta^{n-k}+f_1\theta^{n-k-1},
f_0\theta^{n-k-1},\theta^{n-2-k},\dots,\theta,1.$$
These are dual basis with respect to the pairing from Equation~\eqref{E:dualmap}.
\end{lemma}
\begin{proof}
From Proposition~\ref{P:identifybasis}, we know that the map $\phi: \fIf'_{n-2}\ra \OS$ in Equation~\eqref{E:Iexactlprime}
sends $\zeta_{n-1} \mapsto 1$ and $\theta^i\mapsto 0$ for $0\leq i \leq n-2$.  
The proof of this lemma then has four cases.

{\bf Case 1:} 
We see that $\theta^i \theta^j \stackrel{\phi}{\mapsto} 0 $ if $0\leq i\leq k$ and $0\leq j\leq n-2-k$.

{\bf Case 2:} We compute the image of $(\zeta_{i}+f_i)(f_0\theta^j+\dots+f_{j+k+1-n}\theta^{n-k-1})$ under $\phi$ for $k+1\leq i \leq n-1$ and $n-k-1\leq j \leq n-1$.
We have 
\begin{align*}
 &(\zeta_{i}+f_i)(f_0\theta^j+\dots+f_{j+k+1-n}\theta^{n-k-1})=
 (\zeta_{i}\theta^{n-i}+f_i\theta^{n-i})(f_0\theta^{j+i-n}+\dots+f_{j+k+1-n}\theta^{i-k-1})\\
&=(-f_{i+1}\theta^{n-i-1}-\dots-f_n)(f_0\theta^{j+i-n}+\dots+f_{j+k+1-n}\theta^{i-k-1}).
\end{align*}
Since $n-i-1+j+i-n=j-1\leq n-2$, we see that $$(\zeta_{i}+f_i)(f_0\theta^j+\dots+f_{j+k+1-n}\theta^{n-k-1})\stackrel{\phi}{\mapsto} 0.$$

{\bf Case 3:} We compute the image of $\theta^i  (f_0\theta^j+\dots+f_{j+k+1-n}\theta^{n-k-1})$ under $\phi$ for $0\leq i\leq k$ and $n-k-1\leq j \leq n-1$.
\begin{itemize}
 \item If $i+j\leq n-2$, this maps to 0.
\item If $i+j=n-1$, this maps to 1.
\item If $i+j\geq n$, the product is 
$$f_0\theta^{j+i}+\dots+f_{j+k+1-n}\theta^{n-k-1+i}=-f_{j+k+2-n}\theta^{n-k-2+i}-\dots-f_n \theta^{i+j-n},$$
and since $n-k-2+i\leq n-2$ it maps to 0.
\end{itemize}

{\bf Case 4:} We compute the image of $(\zeta_i+f_i) \theta^j  $ under $\phi$ for $k+1\leq i \leq n-1$ and $0\leq j\leq n-2-k$.
\begin{itemize}
 \item If $i+j\leq n-2$, this maps to 0.
\item If $i+j=n-1$, this maps to 1.
\item If $i+j\geq n$, the product is 
$(\zeta_i+f_i) \theta^j=-f_{i+1}\theta^{j-1}-\dots-f_n \theta^{i+j-n},$
and since $j-1\leq n-2$ it maps to 0.
\end{itemize}

\end{proof}

Finally, it is easy to see in the universal case that the pairing gives an $\fRf$-module homomorphism
$ \fIf_k'\ra 
 \fIf^*_{n-2-k},$ since the pairing factors through multiplication of the fractional ideal elements.

\end{proof}

\begin{corollary}\label{C:JInvdif}
Let $f$ be an $l$-twisted binary $n$-ic form over a base scheme $\BS$. Then
we have an isomorphism of 
$\fRf$-modules
$$\fIf_{n-2}'\isom \mathcal{H}om_{\mathcal{O}_\BS}(\fRf,\OS)$$
given by $ j \mapsto ( r\mapsto \phi(rj))$
where $\phi : \fIf_{n-2}' \ra \OS$ is the map
from Equation~\eqref{E:Iexactlprime}.
\end{corollary}

\section{Main Theorem for $(-1)$-twisted binary forms}\label{S:TwistedMain}

In this section, we will see how a binary form is actually equivalent to a certain combination of the data we have constructed from it.
Let $f$ be a $(-1)$-twisted binary form over a base scheme \BS.
Let $R=\fRf$, let $I=\fIf_{n-3}$, and let $I\ra Q$ be the canonical quotient of 
$\fIf_{n-3}$ from Equation~\eqref{E:Iexactl}. So, $Q\isom V^*$.
From Proposition~\ref{P:rightaction}, we know that
the map $R/\OS\tensor \Sym^{n-3} Q^* \ra Q$
given by the action of $R$ on $I$ and the exact sequence of Equation~\eqref{E:Iexactl}
is identified with the natural map
$ \Sym_{n-2} Q \tensor \Sym^{n-3} Q^* \ra Q$
under the identification $R/\OS \isom \Sym_{n-2} Q$ of
Equation~\eqref{E:Iexactl}.

\begin{definition}
A \emph{binary $n$-pair} is an $\OS$-algebra $R$, an $R$-module $I$, an
exact sequence $0 \ra \Sym^{n-3} Q^*
 \ra I \ra Q \ra 0$ such that $Q$ is a locally free rank $2$ $\OS$-module, and
an isomorphism $R/\OS\isom \Sym_{n-2} Q$ that identifies the map
$R/\OS \tesnor  \Sym^{n-3} Q^* \ra Q$ induced from the action of $R$ on $I$  with the natural map
$\Sym_{n-2} Q \tesnor \Sym^{n-3} Q^*  \ra Q$. 
\end{definition}

\begin{remark}\label{R:n3}
 When $n=3$, we have that $\ker(I\ra Q) \isom \OS$ and the map
$Q \tensor \OS \ra Q$ given by the ring action $R/\OS \tensor \ker (I\ra Q) \ra Q$ is just the natural one.
We can  tensor the exact sequence 
$
0\ra\OS \ra R \ra R/\OS \ra0
$
with $\ker (I\ra Q)$ to show that $R\isom I $ as $R$-modules.
We can conclude that a binary 3-pair is just equivalent to a cubic ring, i.e. an $\OS$-algebra $R$
that is a locally free rank 3 $\OS$-module. 
\end{remark}

There are two equivalent formulations of the definition of a  binary pair that
can be useful.
\begin{proposition}\label{twistedzeroesandones}
 An $\OS$-algebra $R$ and and $R$-module $I$ are in a
a binary pair with $Q$ a free $\OS$-module if and only if 
 $R$ has a $\OS$-module basis $\zeta_0=1,\zeta_1,\dots,\zeta_{n-1}$ and
and $I$ has a $\OS$-module basis $\alpha_1,\alpha_2,\beta_1,\dots,\beta_{n-2}$ such that
$$
\text{the $\alpha_i$ coefficient of $\zeta_j\beta_k$ is }  \begin{cases} 
1 &\text{if $i + j +k = n+1$ }\\
0 &\text{otherwise.}
\end{cases}
$$

\end{proposition}

\begin{proof}
If $Q$ is free with basis $x,y$ and dual basis $\dot{x}$ and $\dot{y}$ , we can explicitly calculate the 
natural map $\Sym_{n-2} Q \tesnor \Sym^{n-3} Q^*  \ra Q$.
Let $\sym(w)$ of a word $w$ be the sum of all distinct permutations of $w$.
We have that
$$
\sym(x^{i} y^{n-2-i}) \tesnor \dot{x}^{j}\dot{y}^{n-3-j} \mapsto
\begin{cases} 
x &\text{if $i=j+1$ }\\
y &\text{if $i =j$ }\\
0 &\text{otherwise.}
\end{cases}
$$
We have $\zeta_j \in \Sym_{n-2} Q$ corresponding to $\sym(x^{n-j-1} y^{j-1})$,
and $\alpha_1$ corresponding to $y$ and $\alpha_2$ corresponding to $x$, and
$\beta_k$ corresponding to $\dot{x}^{k-1}\dot{y}^{n-2-k}$, and we obtain the proposition.
\end{proof}

\begin{proposition}\label{P:exactsequencecond}
An $\OS$-algebra $R$, an $R$-module $I$, a locally free rank 2 \OS-module $Q$ that is a quotient of $I$, and an isomorphism of \OS-modules $\phi :  \Sym_{n-2} Q \isom R/\OS$ 
are in  binary pair if and only if 
$$
\begin{array}{ccccccccc}
0 & \lra  & \Sym_{n-1} Q & \lra & Q\tensor \Sym_{n-2} Q & \lra & 
\left(\ker (I\ra Q)\right)^*\tesnor 
\wedge^2 Q & \lra & 0  \\
&& q_1 q_2 \cdots q_{n-1} &\mapsto& q_1 \tensor q_2 \cdots q_{n-1} &\mapsto&
 && \\
&&  && q \tensor q_1 \cdots q_{n-2} &\mapsto&
\left( k \mapsto q \wedge \phi(q_1\cdots q_{n-2})\circ k \right) && 
\end{array}
$$
is an exact sequence, where $\circ$ denotes the action of $R$ on $I$ followed by the quotient to $Q$.. 
\end{proposition}

Proposition~\ref{P:exactsequencecond} follows from the following Lemma, proven in  Lemma~\ref{L:realexact} of the Appendix (Section~\ref{ModulesAppendix}).
\begin{lemma}\label{realexact}
If $Q$ is any locally free rank 2 $\OS$-module, we have the exact sequence
$$
\begin{array}{ccccccccc}
0 & \lra  & \Sym_{n-1} Q & \lra & Q\tensor \Sym_{n-2} Q & \lra & 
\Sym_{n-3} Q\tesnor 
\wedge^2 Q & \lra & 0 . \\
&& q_1 q_2 \cdots q_{n-1} &\mapsto& q_1 \tensor q_2 \cdots q_{n-1} &\mapsto&
q_2\cdots q_{n-2}\tensor  (q_{n-1} \wedge q_1)&& 
\end{array}
$$
%where the map on the right comes from 
%$\Sym_{n-2} Q \ra  \Sym_{n-3} Q\tesnor Q$.
\end{lemma}

The following lemma is used to construct a $(-1)$-twisted binary form from a binary pair, and is proven in  Lemma~\ref{L:zwd2} of the Appendix (Section~\ref{ModulesAppendix}).

\begin{lemma}\label{L:zwd}
Let $R$ be an \OS-algebra, $I$ be
an $R$-module, $Q$
be a locally free rank 2 \OS-module quotient of $I$, and
$\phi$ be an isomorphism of \OS-modules $\phi :  \Sym_{n-2} Q \isom R/\OS$.
If
$$
\map{\Sym_{n-1} Q \tesnor \ker(I\ra Q)}{\wedge^2 Q}{q_1\cdots q_{n-1} \tensor k}
{q_1\wedge \phi(q_2\cdots q_{n-1})\circ k}
$$
is the zero map, then
$$
\map{\Sym_{n} Q }{\wedge^2 Q}{q_1\cdots q_{n} }
{q_1\wedge \phi(q_2\cdots q_{n-1})\circ \tilde{q_n}}
$$
is well-defined.  Here the $\circ$ denotes the action of $R$ on $I$ followed by the quotient to $Q$
and $\tilde{q}$ denotes a fixed splitting $Q\ra I$.  In particular the map $\Sym_{n} Q \ra \wedge^2 Q$
does not depend on the choice of this splitting.
\end{lemma}
By Proposition~\ref{P:exactsequencecond}, we see that
$\Sym_{n-1} Q \tesnor \ker(I\ra Q)\ra \wedge^2 Q$ is always the zero map for a binary pair, and
thus we can use Lemma~\ref{L:zwd} to construct a $(-1)$-twisted binary form in $\Sym^n Q^* \tesnor \wt Q$ from a binary pair.
We can write the map of Lemma~\ref{L:zwd} as the evaluation 
$$x \mapsto x\wedge \phi(x^{n-2})x$$
 of a degree $n$ map $Q \ra \wedge^2 Q$.  Note this coincides with
the map $x\wedge x^2$ as described in the case of binary cubic forms in \cite[Footnote 3]{HCL2}.

\begin{theorem}\label{T:getback}
Let $(V,f)$ be a $(-1)$-twisted binary form and $(R,I)$ be its associated binary pair.
The $(-1)$-twisted binary form constructed from $(R,I)$ is $f\in\Sym^n V \tesnor \wt V$.
\end{theorem}
\begin{proof}
First we note that the $(-1)$-twisted binary form constructed from $(R,I)$ is
a global section of $\Sym^n V \tesnor \wt V$.
Then, we can check the theorem locally on $S$, so we can assume that $f$ is a free form.
Since $f$ then is a pull-back from the universal form, we can just check the theorem on the universal form $f$
over $\BR=\Z[f_0,\dots,f_n]$. 
Let $x,y$ be the basis of $Q\isom V^*$ and $\dot{x},\dot{y}$ be a corresponding dual basis.

The $(-1)$-twisted binary $n$-ic form associated to our binary $n$-pair is
given by
$$
\map{\Sym_{n} Q }{\wedge^2 Q}{q_1\cdots q_{n} }
{q_1\wedge \phi(q_2\cdots q_{n-1})\circ \tilde{q_n}}.
$$
Thus for $1\leq k \leq n$ we have
\begin{align}\label{E:formformula}
 \sym(x^k y^{n-k}) &\mapsto x \wedge \phi(\sym(x^{k-2} y^{n-k})) x 
		 +x \wedge \phi(\sym(x^{k-1} y^{n-k-1})) y \notag\\
		&+y \wedge \phi(\sym(x^{k-1} y^{n-k-1})) x 
		 +y \wedge \phi(\sym(x^{k} y^{n-k-2})) y \notag\\
=& (\dot{y}(\zeta_{n-k+1}x )+\dot{y}(\zeta_{n-k}y )-\dot{x}(\zeta_{n-k}x ) -\dot{x}(\zeta_{n-k-1}y) )\tensor(x\wedge y) 
\end{align}
where by convention $\sym(x^ay^b)$ is zero if either $a$ or $b$ is negative and $\zeta_i=0$ if $i<1$ or $i>n-1$.
If $K$ is the fraction field of $B$, then the concrete constructions of 
$\fRf$ and $\fIf_{n-3}$ from Section~\ref{S:RI} lie in $\Qf:=K(\theta)/(f_0\theta^n +f_1 \theta^{n-1} +\dots +f_n)$
and are given by Equations~\eqref{E:basis} and \eqref{E:Ibasis}.
From Proposition~\ref{P:identifybasis}, we know we can can identify $x$ with the image of $\zeta_{n-2}$ and $y$ with the image of $\zeta_{n-1}$ in the concrete construction of
$\fIf_{n-3}$.
We can further identify $1, \theta, \dots, \theta^{n-3}$ with the kernel $\Sym^{n-3} Q^*$ of $I \ra Q$.
Using the basis $\zeta_i$ of $\fRf$ and the basis from Equation~\eqref{E:Ibasis} for  $\fIf_{n-3}$,
we have that the $\zeta_{n-1}$ and $\zeta_{n-2}$ coordinates of elements in $\fRf$ and $\fIf_{n-3}$
do not depend on whether taken in the $\fRf$ basis or $\fIf_{n-3}$ basis.  
We can thus compute the expressions $\dot{y}(\zeta_{n-k+1}x ),\dot{y}(\zeta_{n-k}y ),\dot{x}(\zeta_{n-k}x ),\dot{x}(\zeta_{n-k-1}y)$
from Equations~\eqref{E:mult} to prove the proposition.  
\end{proof}

In fact, we have the following theorem, which shows that $(-1)$-twisted binary forms exactly parametrize binary pairs.
\begin{theorem}\label{twistedthm}
For $n\geq 3,$ we have a bijection between $(-1)$-twisted binary $n$-ic forms over $S$ and binary $n$-pairs over $S$, and the bijection commutes with base change in $S$.  In other words, we have a isomorphism of the moduli stack of $(-1)$-twisted binary $n$-ic forms and the moduli stack of binary $n$-pairs.
\end{theorem}

An isomorphism of two $(-1)$-twisted binary $n$-ic forms $f\in \Sym^n V \tesnor \wt V^*$ and $f'\in \Sym^n V' \tesnor \wt (V')^*$
is an isomorphism $V \isom V'$ that preserves $f$.  An isomorphism of two binary $n$-pairs $R$, $I$, $Q$ and $R',$ $I'$, $Q'$
is given by isomorphisms $R \isom R'$, and $I \isom I'$, and $Q \isom Q'$ that respect the exact sequence for $I$ (and $I'$)
and the maps $R/\OS\isom\Sym_{n-2} Q$ and $R'/\OS\isom\Sym_{n-2} Q'$.

 See \cite{BinQuad} for the 
full story for binary quadratic forms.
In the $n=3$ case, from Remark~\ref{R:n3} we know that a binary 3-pair is equivalent to a \emph{cubic ring},
an $\OS$-algebra $R$ such that $R$ is a locally free rank 3 $\OS$-module.
Thus we obtain the following corollary, given in \cite{Delcubic} (see also \cite{Quartic} for a detailed exposition of
this case).

\begin{corollary}
We have a bijection between $(-1)$-twisted binary cubic forms over $S$ and cubic rings over $S$, and the bijection commutes with base change in $S$.  In other words, we have a isomorphism of the moduli stack of $(-1)$-twisted binary $n$-ic forms and the moduli stack of cubic rings.
\end{corollary}

To prove Theorem~\ref{twistedthm} we will rigidify the moduli stacks, and thus we will need to define
based  binary pairs.
\subsection{Based  binary pairs}
A \emph{based  binary pair} is 
 binary pair $R,I,Q$ and a choice of basis $x,y$ of $Q$ such that
$Q$ is the free $\OS$-module on $x$ and $y$.
This gives a natural basis 
of $R/\OS$ as a free rank $(n-1)$ $\OS$-module, and thus $R$ is a free rank $n$ $\OS$-module.
Let $K=(\Sym_{n-3} Q)^*=\ker(I \ra Q)$, and so we have a natural
basis for $K$ as a free rank $n-2$ $\OS$-module.  Thus $I$ is a free rank $n$ $\OS$-module.
However, we do not yet have canonical bases for $R$ and $I$ as $\OS$-modules.  We will pick these using certain normalizations.

Let $\zeta_i=\sym(x^{n-1-i}y^{i-1})$ for $1\leq i\ \leq n-1$ be the given basis of $R/\OS$ and let
$k_j$ for $1\leq j\ \leq n-2$ be the given basis of $K$ dual to the basis
$\sym{x^{j-1}y^{n-2-j}}$ of $\Sym_{n-3} Q$.  Let $\dot{x},\dot{y}\in Q^*$ be a dual basis of $x,y$.
(Recall that $\sym(w)$ for a word $w$ is the sum of all distinct permutations of $w$.)
Thus from Proposition~\ref{twistedzeroesandones}, 
\bq\label{fromsym}
\text{the image of }\zeta_i k_ j \textrm{ in $Q$ is }
\begin{cases} 
x &\text{if $i + j = n-1$ }\\
y &\text{if $i+j=n$}\\
0 &\text{otherwise.}
\end{cases}
\eq
Equation~\eqref{fromsym} allows us to choose normalized lifts of $x$ and $y$ to elements of $I$ that forms a basis along
with the given basis of $K$, and normalized lifts of the $\zeta_i$ to $R$ to form a basis along with $1$.
We choose these lifts so that 
 \bq
\dot{y}(\zeta_i x)=0 \text{ for } 2\leq i \leq n-1
\eq
by changing the lift $x$ by an appropriate multiple of $k_{n-i}$.
We then specify that
 \bq
\dot{x}(\zeta_i x)=0\text{ for } 1\leq i \leq n-1
\eq
by changing the lift of $\zeta_i$ by an appropriate multiple of $1$.
Finally, we specify that
 \bq
\dot{y}(\zeta_i y)=0\text{ for } 2\leq i \leq n-1
\eq
by changing the lift of $y$ by an appropriate multiple of $k_{n-i}$. 
These specifications determine a unique lift of $x$ and $y$ to $I$, and unique lifts of the $\zeta_i$ to $R$, which we will refer to now
as simply $x$, $y$, and $\zeta_i$. 
We will now see that with these choices of normalized bases for $R$ and $I$, we can determine the action of $R$ and $I$ in terms of a small number of variables, and these variables will in fact be the coefficients of the binary form associated to this binary pair.  

There are only $n+1$ coordinates we have not determined in the maps $\zeta_i : I \ra Q$.
Equation~\eqref{fromsym} gives  $\zeta_i : K \ra Q$.  Our choice of normalization gives all but the following. 
Let $-a_{i+1}=\dot{x}(\zeta_i y)$ for $1\leq i \leq n-1$.
Let $a_{0}=\dot{y}(\zeta_1 x)$ and $a_{1}=\dot{y}(\zeta_1 y)$.  These $a_i$ specify the map $\zeta_i : I \ra Q$. We have carefully indexed and signed the $a_i$ so that we have the following.

\begin{proposition}\label{P:formfrompair}
 The $(-1)$-twisted binary form associated to the above based binary pair is
$$\map{\Sym_{n} Q }{\wedge^2 Q}{\sym(x^k y^{n-k}) }
{a_{n-k} x\wedge y}.
$$
\end{proposition}
\begin{proof}
 We use the formula
from Equation~\eqref{E:formformula}.
\end{proof}

Moreover, we find that the coefficients of the associated $(-1)$-twisted binary form determine the based  binary pair.
\begin{proposition}\label{P:determined}
 The maps $\zeta_i : R \ra I$ and $\zeta_i: R \ra R$ are determined by the maps $\zeta_i : I \ra Q$ and the commutativity 
relations on the $\zeta_i$.  Each coordinate of the action and multiplication maps is as a polynomial in the $a_i$ with integral coefficients.
\end{proposition}
\begin{proof}
We view each map $\zeta_i : R\ra I$ as an $n$ by $n$ matrix $Z_i$.  We write $Z_i(a,b)$ for the $a,b$ entry of
$Z_i$, which is the $k_a$ coordinate of $\zeta_i k_b$, where by convention $k_{n-1}=x$ and $k_n=y$.  We let
$\mathcal{K}$ be the set of all of entries of these matrices that are determined by the entries in the last two rows of the matrices
as polynomials in the $a_i$
(i.e. the maps $\zeta_i : I \ra Q$), as well as as all polynomial combinations of the matrix entries which are so determined.
We will show that the systems of equations given by commutativity of the $\zeta_i$ determine all the matrix entries from the
last two rows.
So, by definition we have $Z_i(n-1,k), Z_i(n,k) \in \mathcal{K}$ for $1\leq i \leq n-1$ and $1\leq k \leq n$.

We have two tools that we use to solve for more and more matrix entries. 
\begin{lemma}\label{xstep}
We have
$$
Z_i(n-1-\ell,k)-Z_\ell(n-1-i,k) \in \mathcal{K},\quad \text{ for  $1\leq i \leq n-1$ and $1\leq \ell \leq n-1$}
$$
\end{lemma}
\begin{proof}
Consider the $n-1$st rows ($x$ coordinates) of $Z_i Z_\ell$ and $Z_\ell Z_i $.  
Equating the $j$th entries in both these rows gives
the lemma,
where by convention $Z_i(0,k)=0$.
\end{proof}

\begin{lemma}\label{ystep}
 We have
$$Z_{i}(n-\ell,k)-Z_\ell(n-i,k) \in \mathcal{K,}\quad  \text{ for $1\leq \ell \leq n-1$
and $1\leq i \leq n-1$.}  $$
\end{lemma}
\begin{proof}
Consider the $n$th rows ($y$ coordinates) of $Z_{i} Z_\ell$ and $Z_\ell Z_{i} $.
Equating the $j$th entries in both these rows gives
the lemma.
\end{proof}

We prove, by induction, that all the entries of $Z_i$ are in $\mathcal{K}$ for $1\leq i \leq n-1$.
We can use $i=0$ as the (trivial) base case.  Assuming that all the entries of $Z_i$ are in $\mathcal{K}$, we will
now show that the entries of $Z_{i+1}$ are in $\mathcal{K}$.
Using Lemma~\ref{xstep}, we see that that all matrix entries in the $n-1-i$th row are in $\mathcal{K}$.  (If
$i=0$ this is by from the definition of $\mathcal{K}$.)  Using Lemma~\ref{ystep}, we conclude all the entries
of $Z_{i+1}$ are in $\mathcal{K}$, which completes the induction.

This shows the proposition for the maps $\zeta_i : R \ra I$.  From Equation~\eqref{fromsym},
we see that since $n\geq 3$, each $Z_i$ has a 1 in a matrix entry for which all $z_j$ for $j\ne i$ have entry 0.  
Thus, the action of $R$ on $I$ gives an injection of $R$ into the space on $n$ by $n$ matrices.
To find the $\zeta_k$ coordinate of $\zeta_i\zeta_j$, we just have to look at the matrix entry of
$Z_iZ_j$ where $Z_k$ has a 1 and all $Z_\ell$ for $\ell\ne k$ have a zero.  This shows the proposition for the maps
  $\zeta_i : R \ra I$.
\end{proof}

Now we prove Theorem~\ref{twistedthm}.
\begin{proof}
The stack of binary $n$-pairs
 is the quotient of the stack of based binary $n$-pairs by the $\GL_2$ action given
by change of the basis for $Q$.  Since a based binary $n$-pair is given by $a_0,\dots,a_n\in \OS$,
and we have one such binary pair for every choice of $a_i$'s (given by the corresponding binary form), the moduli
space of based binary $n$-pairs is $\Z[a_0,\dots,a_n]$, and there is a universal based binary $n$-pair.

We have maps between the stack of $(-1)$-twisted binary $n$-ic forms and binary $n$-pairs in both directions, which lift to the rigidified versions of these stacks, the stacks of corresponding based objects.
Theorem~\ref{T:getback} shows that the map from forms to pairs back to forms is the identity.
We will show that the other composition of these constructions is the identity by verifying it
on the rigidified stacks.  
If we start with the universal based binary $n$-pair, Proposition~\ref{P:formfrompair}
shows that the associated form is the universal binary $n$-ic form.  From the 
universal binary $n$-ic form we construct some based binary $n$-pair $(R,I)$, and Proposition~\ref{P:determined}
shows that $(R,I)$ is determined from the binary form constructed from it--which is just the universal binary form
(since we know going from forms to pairs to forms is the identity).  
Since the universal based binary $n$-pair and $(R,I)$ both give the same form, by Proposition~\ref{P:determined}
they are the same.
Thus, we have prove there is an isomorphism of the moduli stack of $(-1)$-twisted binary $n$-ic forms and the moduli stack of binary $n$-pairs.
\end{proof}

We could have done all the work in this section with $\fIf_1$, the dual of 
$\fIf_{n-3}'$, and considered analogs of binary pairs where the conditions on the module
would be $\OS$-dual to the conditions on $I$ in a binary pair.  It turns out some of the constructions
are more natural when working with $\fIf_{n-3}'$ and binary pairs, so we have used that version in this exposition.

One can prove analogs of Theorem~\ref{twistedthm} for all $l$-twisted binary forms.  
We define an \emph{$k$-twisted
binary $n$-pair}  is an $\OS$-algebra $R$, an $R$-module $I$, an
exact sequence $0 \ra \Sym^{n-3} Q^*  \tesnor (\wt Q)^{\tesnor -k}
 \ra I \ra Q \ra 0$ such that $Q$ is a locally free rank $2$ $\OS$-module, and
an isomorphism $R/\OS\isom \Sym_{n-2} Q \tesnor (\wt Q)^{\tesnor k}$ that identifies the map
$R/\OS \tesnor  \Sym^{n-3} Q^*\tesnor (\wt Q)^{\tesnor -k} \ra Q$ induced from the action of $R$ on $I$  with the natural map
$\Sym_{n-2} Q \tesnor (\wt Q)^{\tesnor k} \tesnor \Sym^{n-3} Q^* \tesnor (\wt Q)^{\tesnor -k} \ra Q$.
Given an $l$-twisted binary $n$-ic form, we have an $(l+1)$-twisted binary pair from
$\fRf$, $\fIf_{n-3}'$, and the exact sequence from Equation~\eqref{E:Iexactlprime}.

For example, in a $k$-twisted binary $3$-pair we can see that $I\isom R\tesnor \wt Q^{\tesnor -k}$, by the same argument that we used
to see $I$ was a principal $R$-module in a binary $3$-pair.
 So, we see that $I$ is determined uniquely by $R$ and $Q$.
However, since we have that $R/\OS \isom Q \tesnor (\wt Q)^{\tesnor k}$, we see that not all cubic algebras will appear
as $k$-twisted binary $3$-pairs.

\section{Further questions}
For simplicity, we ask further questions over the base $\Z$.  One naturally wonders which rank $n$ rings appear in a binary pair.  In other words, which rank $n$ rings have modules satisfying the conditions of a binary pair?
  When $n=3$, we saw that the answer is all cubic rings, and each has a unique module and exact sequence that makes a binary pair.  
For $n=4$, there is another characterization of the answer.  In \cite{BinQuartic}
it is shown that the quartic rings associated to binary quartic forms are exactly the quartic rings with monogenic cubic resolvents.  The 
cubic resolvent is a certain integral model of the classical cubic resolvent field.  Are there such connections with resolvents for higher $n$?

Simon \cite{Simon} asks which maximal orders are constructed from binary $n$-ic forms.  He defines the \emph{index} of a form to be the index of
its ring in the maximal order.  He begins a program to compute all forms with a given index.  For example, in the quartic case he uses elliptic curves to compute the forms of index 1 and a certain $I$ and $J$ ($\GL_2(\Z)$ invariants of a binary quartic form).  Simon also shows
that there are no index 1 forms with a root generating a cyclic extension of prime degree at least 5.  
 In general, it would be very interesting to understand which maximal orders are associated to binary forms.

\appendix
\section{Verifications of $\Z$ basis of $\If^k$ }\label{S:App1}

\begin{proposition}\label{P:Ibasis}
For $f$ with $f_0\ne 0$ and $1 \leq k\leq n-1$, the \Rf module $\If^k$ is a free
rank $n$ $\Z$-module on the basis
given in Equation~\ref{E:Ibasis}. 
\end{proposition}

\begin{lemma}\label{L:Ibasis}
We have
$$
\Rf\theta^k \subset \ang{\Rf,\theta,\theta^2,\dots,\theta^k}_\Z
$$
for all $k\geq 1$.
\end{lemma}
\begin{proof}[of Lemma~\ref{L:Ibasis}]
We see that
\begin{align*}
\zeta_i \theta^k=f_0\theta^{k+i}+\dots+f_{i-1}\theta^{k+1}\quad\textrm{if } k+i\leq n-1
\end{align*}
and
\begin{align*}
\zeta_i \theta^k=&\theta^{k+i-n}(f_0\theta^{n}+\dots+f_{i-1}\theta^{n-i+1})\quad\textrm{if } k+i\geq n\\
=&-\theta^{k+i-n}(f_i\theta^{n-i}+\dots+f_{n})\\
=&-(f_i\theta^{k}+\dots+f_{n}\theta^{k+i-n}).
\end{align*}
\end{proof}

\begin{proof}[of Proposition~\ref{P:Ibasis}]
So, as a $\Z$-module
$\If^k$ is generated by
$1,\theta,\dots,\theta^k,\zeta_{k+1},\dots\zeta_{n-1}$
for $k \geq 1$.
If $k\leq n-1$, then since $f_0\ne 0$, we have that $1,\theta,\dots,\theta^k,\zeta_{k+1},\dots\zeta_{n-1}$ 
generate a free $\Z$-module, 
and thus are a $\Z$-module basis for $\If^k$.
\end{proof}

\begin{proposition}\label{P:Isbasis}
The $\Z$-module $\If^\#$ defined by Equation~\eqref{E:Isbasis}
is an ideal. 
\end{proposition}
\begin{proof}
Let $J=\theta\If^\#=\ang{\zeta_1,\zeta_2,\dots,\zeta_{n-1}, -f_n}_\Z$.
From the multiplication table given in
Equation~\eqref{E:mult} we see that
   $\ang{\zeta_1,\dots,\zeta_{n-1}}_\Z\cdot \ang{\zeta_1,\dots,\zeta_{n-1}}_\Z
   \subset J$.
   Thus, $\Rf J \subset J$ and so $J$ and thus $\If^\#$ are ideals
   of $\Rf$.
\end{proof}

\begin{proposition}\label{P:invertextra}
Let $f$ be a non-zero binary $n$-ic form.
Then, the fractional ideal $\If$ is invertible if and only if the form $f$ is primitive.
Also, the fractional ideal $\If^\#$ is invertible if and only if the form $f$ is primitive.
We always have that $\If^\#=(\Rf:\If)$,
where $(A:B)=\{x\in Q_f |\, xB \subset A\}$.
  In the
case that $f$ is primitive, $\If^{-1}=\If^\#$.
\end{proposition}
\begin{proof}
First, we act by $\GZ$ so that we may assume $f_0\ne 0$.
Since $\If^\#\subset\Rf$ and $\theta\If^\#=\ang{\zeta_1,\zeta_2,\dots,\zeta_{n-1}, -f_n}_\Z
\subset\Rf$, we have $\If\If^\#\subset \Rf$. 
More specifically, we see that
\begin{align*}
\If\If^\#=&\ang{f_0,\zeta_1+f_1,\dots,\zeta_{n-1}+f_{n-1},\zeta_1,\zeta_2,\dots,\zeta_{n-1},
 -f_n}_\Z\\
 =&\ang{f_0,f_1,\dots,f_{n},\zeta_1,\zeta_2,\dots,\zeta_{n-1} }_\Z,
\end{align*}
which is equal to $\Rf$ if and only if the form $f$ is primitive.

Let $x\in (\Rf:\If)$.  Since $1\in\If$, we have $x\in \Rf$.
Write $x=x_0+\sum_{i=0}^{n-1} x_i(\zeta_i+f_i)$ where the $x_i\in\Z$.
Also, $\theta x\in \Rf$, and
 $\theta x = x_0\theta+\sum_{i=0}^{n-1} x_i\zeta_{i+1}$.
Thus $f_0\mid x_0$, which implies $x\in \If^\#$.
We conclude $\If^\#=(\Rf:\If)$.  

Suppose $\If$ is invertible.  Then, its inverse is $(\Rf:\If)=\If^\#$, which implies
$\If\If^\#=\Rf$ and the form $f$ is primitive.  Suppose $I_f^\#$ is invertible, then
the norm of $\If\If^\#$ is the product of the norms of $\If$ and $\If^\#$, which is 1.
Since $\If\If^\#\subset R_f$, we have that $\If\If^\#=\Rf$ and the form $f$ is primitive.
\end{proof}

\section{Maps between locally free $\OS$-modules}\label{ModulesAppendix}
Let $S$ be a scheme.  In this appendix we will give several basic facts about 
maps between locally free $\OS$-modules.

\begin{lemma}\label{L:Sym}
Let $V$ be a locally free $\OS$ module.  We have $(\Sym_n V)^* \cong \Sym^n V^*$.
\end{lemma}

\begin{lemma}\label{L:Syminside}
Let $V$ be a locally free $\OS$ module. Inside of $V^{\tesnor a+b}$ the submodule $\Sym_{a+b} V$  is a submodule of
$\Sym_a V \tesnor \Sym_b V$.  Thus we have a natural map
$$
\Sym_{a+b} V \ra \Sym_a V \tesnor \Sym_b V,
$$
which is injective.
\end{lemma}

\begin{lemma}\label{P:symline}
If $L$ is a locally free rank 1 $\OS$-module and $V$ is a locally free rank $n$
$\OS$-module, then $\Sym^k(V\tensor L)\isom\Sym^k V \tensor L^{\tesnor k}$.
\end{lemma}

\begin{lemma}\label{L:rk2}
If $V$ is a locally free $\OS$-module of rank two then $V \tensor \wedge^2 V^* \isom V^*$.
\end{lemma}

\begin{lemma}\label{L:realexact}
If $Q$ is any locally free rank 2 $\OS$-module, we have the exact sequence
$$
\begin{array}{ccccccccc}
0 & \lra  & \Sym_{n-1} Q & \lra & Q\tensor \Sym_{n-2} Q & \lra & 
\Sym_{n-3} Q\tesnor 
\wedge^2 Q & \lra & 0 . \\
&& q_1 q_2 \cdots q_{n-1} &\mapsto& q_1 \tensor q_2 \cdots q_{n-1} &\mapsto&
q_2\cdots q_{n-2}\tensor  (q_{n-1} \wedge q_1)&& 
\end{array}
$$
%where the map on the right comes from 
%$\Sym_{n-2} Q \ra  \Sym_{n-3} Q\tesnor Q$.
\end{lemma}

\begin{proof}
 We can check this sequence is exact and thus on free $Q$ generated by $x$ and $y$.
For a word $w$ in $x$ and $y$, let $\sym(w)$ denote the sum of all distinct permutations of $w$.
Then, a basis for $\Sym_{n-1} Q$ is $\alpha_k=\sym(x^k y^{n-1-k})$ for $0\leq k \leq n-1$.
A basis for $Q\tensor \Sym_{n-2} Q$ is given by
\begin{align*}
\beta_0&= y\tensor\sym(y^{n-2})\\
\beta_k&=x\tensor\sym(x^{k-1}y^{n-1-k})+y\tensor\sym(x^{k}y^{n-2-k}) \text{ for } 1\leq k\leq n-2 \\
\beta_{n-1}&= x\tensor\sym(x^{n-2})\\
\gamma_\ell&= x\tesnor\sym(x^\ell y^{n-2-\ell})\text{ for } 0\leq \ell\leq n-3. 
\end{align*}
We see that in the sequence of the proposition, $\alpha_i \mapsto \beta_i$ and the 
$\gamma_\ell$ map to a basis of $\Sym_{n-3} Q\tesnor 
\wedge^2 Q$.
\end{proof}

\begin{lemma}\label{L:zwd2}
Let $R$ be an \OS-algebra, $I$ be
an $R$-module, $Q$
be a locally free rank 2 \OS-module quotient of $I$, and
$\phi$ be an isomorphism of \OS-modules $\phi :  \Sym_{n-2} Q \isom R/\OS$.
If
$$
\map{\Sym_{n-1} Q \tesnor \ker(I\ra Q)}{\wedge^2 Q}{q_1\cdots q_{n-1} \tensor k}
{q_1\wedge \phi(q_2\cdots q_{n-1})\circ k}
$$
is the zero map, then
$$
\map{\Sym_{n} Q }{\wedge^2 Q}{q_1\cdots q_{n} }
{q_1\wedge \phi(q_2\cdots q_{n-1})\circ \tilde{q_n}}
$$
is well-defined.  Here the $\circ$ denotes the action of $R$ on $I$ followed by the quotient to $Q$
and $\tilde{q}$ denotes a fixed splitting $Q\ra I$.  In particular the map $\Sym_{n} Q \ra \wedge^2 Q$
does not depend on the choice of this splitting.
\end{lemma}
\begin{proof}
Since $\Sym_{n-1} Q  \subset  Q\tensor \Sym_{n-2} Q$
as submodules of $Q^{\tensor n}$ (see Lemma~\ref{L:Syminside}), the first map $\Sym_{n-1} Q \tesnor \ker(I\ra Q)\ra \wedge^2 Q$ is
well-defined.  For a given choice
of splittings $\Sym_{n-2} Q \ra R$ and $Q \ra I$, consider the following commutative diagram.
$$
\xymatrix{
   & \Sym_n Q \ar[ld] \ar[d] \ar[rd]&\\
\Sym_{n-2} Q \tesnor \Sym_2 Q \ar[d] &Q\tensor \Sym_{n-2} Q \tensor Q \ar[dd] & \Sym_{n-1} Q\tensor Q \ar[d] \\
R\tensor \Sym_2 Q \ar[d] \ar[rd] & &\Sym_{n-1} Q \tensor I \ar[d] \\
Q \tensor R \tensor Q \ar[r] &Q \tensor R \tensor I \ar[d] &Q \tensor \Sym_{n-2} Q \tensor I \ar[l] \\
 &Q\tensor Q \ar[d] & \\
 &\wedge^2 Q &
}
$$
To investigate the effect of a different splitting $Q\ra I$ on the map
$\Sym_n Q \ra \wedge^2 Q$, we take the route on the right hand side of the diagram.  The difference
between the composite maps from two different splittings will land in the submodule
$\Sym_{n-1} Q \tensor \ker(I \ra Q)$ of the $\Sym_{n-1} Q \tensor I$ term, and thus be zero in the final map by the hypothesis
of the lemma. 

To investigate the effect of a different splitting $\Sym_{n-2} Q \isom R/\OS \ra R$ on the map
$\Sym_n Q \ra \wedge^2 Q$, we take the route on the left hand side of the diagram.
The difference between the maps from the different splittings will land in the submodule
$\OS \tensor \Sym_2 Q$ of the $R \tensor \Sym_2 Q$ term, and it is easy to see that the difference will be zero
in the composite map.
\end{proof}

\section{Acknowledgements}
The author would like to thank Manjul Bhargava for asking the questions that inspired this research, guidance along the way, and helpful feedback both on the ideas
and the exposition in this paper.  She would also like to thank Lenny Taelman for suggestions for improvements to the paper.  This work was done as part of the author's Ph.D. thesis at Princeton University, and during the work she was supported by an NSF Graduate Fellowship, an NDSEG Fellowship, an AAUW Dissertation Fellowship, and a Josephine De K\'{a}rm\'{a}n Fellowship.  This paper was prepared for submission while the author was supported by an American Institute of Mathematics Five-Year Fellowship.

\end{document}